\documentclass[a4paper]{amsart}
\usepackage[T1]{fontenc}
\usepackage[utf8]{inputenc}
\usepackage{color}
\usepackage{amsthm}
\usepackage{amstext}
\usepackage{amssymb}
\usepackage{stmaryrd}
\usepackage{esint}

\makeatletter

\pdfpageheight\paperheight
\pdfpagewidth\paperwidth

\numberwithin{equation}{section}
\numberwithin{figure}{section}
\theoremstyle{plain}
\newtheorem{thm}{\protect\theoremname}[section]
  \theoremstyle{remark}
  \newtheorem*{rem*}{\protect\remarkname}
  \theoremstyle{plain}
  \newtheorem{prop}[thm]{\protect\propositionname}
  \theoremstyle{definition}
  \newtheorem{defn}[thm]{\protect\definitionname}
  \theoremstyle{plain}
  \newtheorem{cor}[thm]{\protect\corollaryname}
  \theoremstyle{plain}
  \newtheorem{lem}[thm]{\protect\lemmaname}
  \theoremstyle{definition}
  \newtheorem{example}[thm]{\protect\examplename}

\usepackage{amsfonts}
\usepackage{amsthm}
\usepackage{cite}
\usepackage{url}\usepackage{pb-diagram}

\makeatother

  \providecommand{\corollaryname}{Corollary}
  \providecommand{\definitionname}{Definition}
  \providecommand{\examplename}{Example}
  \providecommand{\lemmaname}{Lemma}
  \providecommand{\propositionname}{Proposition}
  \providecommand{\remarkname}{Remark}
\providecommand{\theoremname}{Theorem}

\begin{document}

\title{On a Hitchin-Thorpe inequality for manifolds with foliated boundaries}

\author{Ahmed J. Zerouali}
\begin{abstract}
We prove a Hitchin-Thorpe inequality for noncompact 4-manifolds with
foliated geometry at infinity by extending on previous work by Dai
and Wei. After introducing the objects at hand, we recall some preliminary
results regarding the $G$-signature formula and the rho invariant,
which are used to obtain expressions for the signature and Euler characteristic
in our geometric context. We then derive our main result, and present
examples.

\medskip

\noindent R\'esum\'e: En se basant sur des travaux de Dai et Wei,
on d\'emontre une in\'egalit\'e de Hitchin-Thorpe pour vari\'et\'e
non-compactes de dimension 4, et munies d'une g\'eom\'etrie feuillet\'ee
\`a l'infini. Apr\`es avoir d\'efini les notions pertinentes \`a
cette \'etude, on rappelle quelques r\'esultats concernant la formule
de $G$-signature et l'invariant rho, qu'on utilise ici pour obtenir
des expressions de la signature et de la caract\'eristique d'Euler
dans notre cadre g\'eom\'etrique. On d\'emontre ensuite nos r\'esultats
principaux avant de pr\'esenter quelques exemples.
\end{abstract}


\maketitle

\section{\textbf{Introduction}}

If $M$ is a closed compact and oriented 4-dimensional Einstein manifold,
then its Euler characteristic $\chi(M)$ and its (Hirzebruch) signature
$\tau(M)$ must satisfy $\chi(M)\geq3|\tau(M)|/2$. This is the statement
of the original Hitchin-Thorpe inequality, which was proved in \cite{[H]}
and \cite{[ST]}, and subsequently extended to various contexts, as
in \cite{[Grom82]}, \cite{[Kot98]}, \cite{[Kot12]} and \cite{[Sam98]}
to name a few. The most relevant generalization for the present work
is that of Dai and Wei in \textcolor{black}{\cite{[DaW]}, where the
manifolds of interest are noncompact and have fibred geometries at
infinity.}

In the present note, we derive a Hitchin-Thorpe inequality for manifolds
with a foliation structure (resolved by a fibration) at infinity.
To be more precise, if $\overline{M}$ is the compactification at
infinity of such a space $M$, \textcolor{black}{the type of boundary
$W=\partial\overline{M}$ that we are considering is obtained as a
quotient $W=\widetilde{W}/\Gamma$, where $\widetilde{W}$ is the
total space of a smooth fibration} \textcolor{black}{$F\rightarrow\widetilde{W}\overset{\phi}{\rightarrow}\Sigma$,
and $\Gamma$ is a finite group acting smoothly and freely on $\widetilde{W}$,
and also acting smoothly on the base $\Sigma$ in such a way that
the projection $\phi:\widetilde{W}\rightarrow\Sigma$ is $\Gamma$-equivariant:
$\phi(g\cdot w)=g\cdot\phi(w)$ for all $w\in\widetilde{W}$ and $g\in\Gamma$.
If $\nu:\widetilde{W}\rightarrow\widetilde{W}/\Gamma$ is the quotient
map with respect to the group action, then the fibration $\widetilde{W}$
induces a foliation atlas $\mathcal{F}$ on the boundary $W=\widetilde{W}/\Gamma$
whose leaves are the images of the fibres $F$ under $\nu$.}

In addition to the foliation on $W$, we assume that $\overline{M}$
is endowed with a $\mathcal{F}$-metric or a foliated cusp metric,
which we now introduce. Let $c:I\times W\rightarrow M$ be a local
diffeomorphism describing a tubular neighbourhood of the boundary
(with $I=[0,1]$), $x\in\mathcal{C}^{\infty}(M,\mathbb{R}_{+})$ a
boundary defining function for $W$, and $\hat{\nu}:=\text{Id}_{I}\times\nu$
be the covering $I_{x}\times\widetilde{W}\rightarrow I_{x}\times W$.
A metric \ensuremath{g_{\mathcal{F}}}
 on \ensuremath{M}
 is called a (product-type) $\mathcal{F}$-metric (\cite{[R]}, p.1314)
if it satisfies 
\[
\hat{\nu}^{*}c^{*}g_{\mathcal{F}}=\frac{dx^{2}}{x^{4}}+\frac{\phi^{*}h}{x^{2}}+\kappa\in\mathcal{C}^{\infty}\big(I\times\widetilde{W},S^{2}T^{*}(I\times\widetilde{W})\big),
\]
where $h$ is a $\Gamma$-invariant metric on $\Sigma$, and $\kappa\in\mathcal{C}^{\infty}(\widetilde{W},S^{2}T^{*}\widetilde{W})$
is a $\Gamma$-invariant tensor restricting to a metric on the fibres
of $\widetilde{W}\rightarrow\Sigma$. With the same notations, a metric
$g_{\mathcal{F}c}$ on $M$ is a (product-type) foliated cusp metric
(\cite{[R]}, (1.6) p.1314) if it is such that 
\[
\hat{\nu}^{*}c^{*}g_{\mathcal{F}c}=\frac{dx^{2}}{x^{2}}+\phi^{*}h+x^{2}\kappa\in\mathcal{C}^{\infty}\big(I\times\widetilde{W},S^{2}T^{*}(I\times\widetilde{W})\big).
\]
Manifolds admitting such metrics are the main topic of \cite{[R]},
where a pseudo-differential calculus adapted to this setting is constructed,
before addressing the index theory of the associated Dirac-type operators.
More recently, a study of their Hodge cohomology has been carried-out
in \cite{[GRR]}.

We start our treatment by recalling some facts pertaining to the $G$-signature
formula and the Atiyah-Bott-Lefschetz fixed-point formula in the setting
of manifolds with boundaries, on which a finite group acts smoothly
and has only isolated fixed-points in the interior. The results are
then established in the third section, first by obtaining the signature
and the Euler characteristic for a noncompact $4k$-dimensional manifold
with a metric which is asymptotic to a foliated boundary or a foliated
cusp metric, before specializing to the four dimensional case where
\ensuremath{\widetilde{W}}
 is a circle bundle over a compact Riemann surface, and finally deriving
the Hitchin-Thorpe inequality. In the last section we illustrate the
results with some examples.

\section{\textbf{Background Material}}

We give a brief exposition of the results needed for the upcoming
sections.

\subsection{The G-Signature Formula\textmd{\label{subsec2.1}}}

The $G$-signature formula is a special case of the Lefschetz fixed
point formula for elliptic complexes. In this section, $(\widetilde{X},g)$
is a $4k$-dimensional compact oriented Riemannian manifold with nonempty
boundary $\widetilde{W}$, and a metric $g$ which is of product type
near the boundary. Let $\Gamma\subset\text{Isom}(\widetilde{X})$
be a finite group of orientation-preserving isometries such that the
action has only isolated fixed points in the interior of $\widetilde{X}$.
We employ the following notations:
\begin{itemize}
\item $^{g}\widetilde{B}$ is the (even) signature operator on $\widetilde{X}$,
and $^{g}\widetilde{A}$ is the odd signature operator on $\widetilde{W}$
(induced by $g_{|\widetilde{W}}$).
\item For an eigenvalue $\lambda\in\text{Spec}\phantom{}^{g}\widetilde{A}\smallsetminus\{0\}$
($^{g}\widetilde{A}$ is self-adjoint and elliptic), its associated
eigenspace is $E_{\lambda}\subset L^{2}(\widetilde{W},\Lambda T^{*}\widetilde{W})$.
\item For $r=0,\cdots,m=4k$, $\widehat{H}^{r}$ denotes the image of the
relative cohomology group $H^{r}(\widetilde{X},\widetilde{W};\mathbb{C})$
in the absolute cohomology $H^{r}(\widetilde{X};\mathbb{C})$. Also,
$\widehat{H}_{\pm}^{2k}$ are the subspaces of $\widehat{H}^{2k}$
on which the non-degenerate bilinear form induced by the cup product
is positive or negative definite.
\item For a fixed element $a\in\Gamma$, $a^{*}|_{\widehat{H}^{r}}$ is
the morphism induced in cohomology, \ensuremath{a_{\lambda}^{*}}
 is the morphism induced on $E_{\lambda}$, and $a_{*}|_{x}:T_{x}\widetilde{X}\rightarrow T_{a(x)}\widetilde{X}$
is the differential at $x\in\widetilde{X}$.
\end{itemize}
Let $\phi:\widetilde{X}\rightarrow\widetilde{X}$ be an isometry that
has only isolated fixed-points, and $dVol^{g}$ the volume form associated
to $(\widetilde{X},g)$. For a fixed-point $x\in\widetilde{X}$ of
$\phi$, the differential $\phi_{*|x}\in\text{Aut}(T_{x}\widetilde{X})$
is also an isometry, and we have an orthogonal decomposition $T_{x}\widetilde{X}=\oplus_{j=1}^{2k}V_{j}$
into $\phi$-invariant 2-planes. We choose an orthogonal basis $\{e_{j},e'_{j}\}$
of $V_{j}$ for all $j=1,\cdots,2k$ such that 
\[
(dVol^{g})_{x}(e_{1},e'_{1},\cdots,e_{2k},e'_{2k})=1,
\]
With respect to this basis, $\phi_{*|x}$ is block-diagonal with components
of the form
\[
\left(\phi_{*|x}\right)_{|V_{j}}=\left(\begin{array}{cc}
\cos\left[\theta_{\phi,j}(x)\right] & -\sin\left[\theta_{\phi,j}(x)\right]\\
\sin\left[\theta_{\phi,j}(x)\right] & \cos\left[\theta_{\phi,j}(x)\right]
\end{array}\right).
\]
The numbers $\{\theta_{\phi,j}(x)\}_{j=1}^{2k}$ constitute a \textit{coherent
system of angles} for $\phi_{*|x}$ (\cite{[AB2]}, p.473).

Now, for a fixed isometry $a\in\Gamma$, we define the following topological
quantities:
\begin{itemize}
\item The $G$\textit{-signature} with respect to $a\in\Gamma$ is given
by: 
\[
\tau(a,\widetilde{X})=\text{Tr}\left(a^{*}|_{\widehat{H}_{+}^{2k}}\right)-\text{Tr}\left(a^{*}|_{\widehat{H}_{-}^{2k}}\right).
\]

\item The $G$-\textit{eta function} with respect to $a$ is 
\[
\eta_{a}(s,\phantom{}^{g}\widetilde{A})={\displaystyle \sum_{\lambda\in\text{Spec}\phantom{}^{g}\widetilde{A}\smallsetminus\{0\}}sgn\lambda|\lambda|^{-s}\text{Tr}(a_{\lambda}^{*}).}
\]
This function is holomorphic for $Re(s)\gg0$ and admits a meromorphic
continuation to the whole complex plane with no pole at $s=0$ (\cite{[Don]}).
The associated $G$-eta invariant is $\eta_{a}(\phantom{}^{g}\widetilde{A})\equiv\eta_{a}(0,\phantom{}^{g}\widetilde{A})$.
For $a=\text{Id}$, this is the usual eta invariant.
\item Let $\text{Fix}(a)\subset\widetilde{X}$ be the fixed point set of
$a\in\Gamma$. For $z\in\text{Fix}(a)$, the signature defect at that
point is given by 
\[
\text{def}(a,\phantom{}^{g}\widetilde{B})[z]=\prod_{j=1}^{2k}\left(\frac{\lambda_{j}+1}{\lambda_{j}-1}\right)=\prod_{j=1}^{2k}(-i)\cot\left(\theta_{a,j}(z)/2\right),
\]
where $\lambda_{j}=\exp\left(i\theta_{a,j}(z)\right)$ is the \textit{j}-th
eigenvalue of the linear map $a_{*}|_{z}\in\text{Aut}(T_{z}\widetilde{X})$
(Theorem 4.5.2 \cite{[G]}, formula (7.2) \cite{[AB2]}).
\end{itemize}
As a special case of a theorem proved by Donnelly (\cite{[Don]},
Theorem 2.1; \cite{[G]} Theorem 4.5.8 ), we have the $G$-signature
formula:
\begin{thm}
Under the hypotheses of this subsection, the G-signature formula for
\ensuremath{a\in\Gamma\smallsetminus\{\text{Id}\}}
 reads 
\[
\tau(a,\widetilde{X})=\sum_{z\in\text{Fix}(a)}\text{def}(a,\phantom{}^{g}\widetilde{B})[z]-\frac{1}{2}\eta_{a}(\phantom{}^{g}\widetilde{A}).
\]
\end{thm}
\begin{rem*}
In \cite{[APS1]} and \cite{[Don]}, the authors use the restriction
of $\phantom{}^{g}\widetilde{A}$ to even forms rather than the odd
signature operator itself, which is why they do not have this extra
$(1/2)$ factor that appears here and in\cite{[DaW]} next to the
eta invariants. 
\end{rem*}

\subsection{The rho invariant of a finite covering\textmd{\label{subsec2.2}}}

With the same notations as above, $\Gamma$ acts smoothly and freely
on $(\widetilde{W},g_{|\widetilde{W}})$, and $\nu:\widetilde{W}\rightarrow W=\widetilde{W}/\Gamma$
is the induced covering projection. In this subsection, $A_{g}$ denotes
the odd signature operator associated to the metric induced by $g\in\mathcal{C}^{\infty}(\widetilde{X},S^{2}T^{*}\widetilde{X})$
on $W$ ($\phantom{}^{g}\widetilde{A}=\nu^{*}A_{g}$).

Atiyah, Patodi and Singer studied the signature with local coefficients
in \cite{[APS2]}. For a finite covering \ensuremath{\widetilde{W}\overset{\nu}{\rightarrow}W}
, and a one-dimensional unitary representation $\alpha\text{ of }\Gamma$
with associated flat bundle $E_{\alpha}\rightarrow W$, they obtain
the following expression ((I.5) in \cite{[Don]}): 
\[
\eta(A_{g,\alpha})=\frac{1}{|\Gamma|}\sum_{a\in\Gamma}\eta_{a}(\phantom{}^{g}\widetilde{A})\chi_{\alpha}(a),
\]
where $A_{g,\alpha}$ is the operator induced by $A_{g}$ on $\Lambda T^{*}W\otimes E_{\alpha}$
and $\chi_{\alpha}(a)$ is the character of $a\in\Gamma$. If $\alpha$
is the trivial representation, this formula reduces to ((I.6) \cite{[Don]})
\[
\eta(A_{g})=\frac{1}{|\Gamma|}\sum_{a\in\Gamma}\eta_{a}(\phantom{}^{g}\widetilde{A}).
\]
The \textit{rho invariant} of the finite covering $\widetilde{W}\overset{\nu}{\rightarrow}W$
is then defined as: 
\[
\rho(\widetilde{W},W)=\frac{1}{2}\left[\eta(\phantom{}^{g}\widetilde{A})-|\Gamma|\eta(A_{g})\right]=-\frac{1}{2}\sum_{a\neq Id}\eta_{a}(\phantom{}^{g}\widetilde{A}).
\]
Eta invariants depend on the choice of a Riemannian metric, but this
is not the case for the rho invariant of a finite covering:
\begin{prop}
Let $g_{i}\text{, }i=0,1$ be two Riemannian metrics on $W$, $\tilde{g}_{i}=\nu^{*}g_{i}$
the pullback metrics on $\widetilde{W}$, and $\rho_{i}(\widetilde{W},W)$
the corresponding rho invariants; then 
\[
\rho_{0}(\widetilde{W},W)=\rho_{1}(\widetilde{W},W).
\]
\end{prop}
\begin{proof}
Let $\widetilde{M}=\widetilde{W}\times I\text{, }M=W\times I\text{, }\widetilde{\nu}=\nu\times\text{Id}_{I}$
where $I=[0,1]\subset\mathbb{R}$. For $s\in I$, we have a metric
on $M$ given by $h=(1-s)g_{0}+sg_{1}+ds^{2}$ with pullback $\widetilde{h}=\widetilde{\nu}^{*}h$.
If $\widetilde{A}_{i}\text{ and }A_{i}$ are the odd signature operators
associated to $\widetilde{g}_{i}\text{ and }g_{i}$, then applying
the Atiyah-Patodi-Singer theorem to \ensuremath{\widetilde{M}\text{ and }M}
 yields
\begin{align*}
\tau(\widetilde{M})= & \int_{\widetilde{M}}L(\widetilde{M})-\frac{1}{2}[\eta(\widetilde{A}_{1})-\eta(\widetilde{A}_{0})],\\
\tau(M)= & \int_{M}L(M)-\frac{1}{2}[\eta(A_{1})-\eta(A_{0})].
\end{align*}
 From the definition of the $\rho_{i}$ and these formulas, we have
\[
\rho_{1}(\widetilde{W},W)-\rho_{0}(\widetilde{W},W)=\left[\int_{\widetilde{M}}L(\widetilde{M})-|\Gamma|\int_{{M}}L({M})\right]-\tau(\widetilde{M})+|\Gamma|\tau({M})
\]
The first term in the RHS vanishes since $\widetilde{\nu}$ is a local
isometry. To see that $\tau(\widetilde{M})$ vanishes, we consider
the long exact sequence of relative cohomology 
\[
\cdots\overset{\delta^{*}}{\rightarrow}H^{2k}(\widetilde{M},\partial\widetilde{M})\overset{j}{\rightarrow}H^{2k}(\widetilde{M})\overset{i}{\rightarrow}H^{2k}(\partial\widetilde{M})\overset{\delta^{*}}{\rightarrow}H^{2k+1}(\widetilde{M},\partial\widetilde{M})\overset{j}{\rightarrow}\cdots.
\]
Since $\partial\widetilde{M}=\widetilde{W}\sqcup-\widetilde{W}$,
$H^{2k}(\widetilde{M})=H^{2k}(\widetilde{W})$ and $H^{2k}(\partial\widetilde{M})=H^{2k}(\widetilde{W})\oplus H^{2k}(\widetilde{W})$,
the map $i$ is injective, and the image of $H^{2k}(\widetilde{M},\partial\widetilde{M})$
in $H^{2k}(\widetilde{M})$ is trivial, which means that $\tau(\widetilde{M})=0$.
With the same argument we get $\tau(M)=0$, from which $\rho_{1}=\rho_{0}$
follows.
\end{proof}

\subsection{Manifolds with foliated boundaries}

Let $M$ be a connected, oriented and noncompact manifold of dimension
$4k$, $k\ge1$.
\begin{defn}
\label{Def_mfld_foliated_bndry}We say that a complete Riemannian
manifold $(M,g)$ has a \textit{foliated geometry at infinity resolved
by a fibration} if there exists a compactification at infinity $\overline{M}=M\cup W$
by a boundary $W=\partial\overline{M}$, and a field of symmetric
bilinear forms $\bar{g}:\overline{M}\rightarrow S^{2}T^{*}\overline{M}$
satisfying the following conditions:\\
(i) The bilinear form $\bar{g}$ coincides with the initial metric
on the interior: $\bar{g}_{|M}=g\in\Gamma(M,S^{2}T^{*}M)$;\\
(ii) The boundary $W$ is the base of a finite covering $\nu:\widetilde{W}\rightarrow W=\widetilde{W}/\Gamma$,
where $\Gamma$ is a finite group of orientation-preserving isometries
acting freely on $\widetilde{W}$;\\
(iii) The manifold $\ensuremath{\widetilde{W}}$ is also the total
space of a smooth fibration $F\rightarrow\widetilde{W}\overset{\phi}{\rightarrow}\Sigma$,
where $F$ and $\Sigma$ are closed compact oriented manifolds with
$\dim F>0$;\\
(vi) The group $\Gamma$ acts smoothly on the base $\Sigma$, and
the projection $\phi:\widetilde{W}\rightarrow\Sigma$ is $\Gamma$-equivariant.
\end{defn}
Since we will only consider foliations resolved by a fibration in
this paper, we will simply say that $M$ has a \textit{foliated geometry
at infinity} when it satisfies the definition above, and that the
compactification $\overline{M}$ has a \textit{foliated boundary}.
Under these assumptions, we will use $\mathcal{F}$ to denote the
foliation of the boundary $W=\partial\overline{M}$ by the images
of the fibres $F$ under the projection $\nu:\widetilde{W}\rightarrow W$.

For the remaining of this section, we let $\overline{M}$ be a manifold
with foliated boundary, and we fix a boundary defining function (bdf)
$x\in\mathcal{C}^{\infty}(M,\mathbb{R}_{+})$ for $\overline{M}$,
so that $W=\{x=0\}$ and $x>0$ on $M=\overline{M}\smallsetminus W$.
We will consider two classes of foliated cusp vector fields on $\overline{M}$,
namely the $\mathcal{F}$-vector fields:
\[
\mathcal{V}_{\mathcal{F}}(\overline{M}):=\left\{ \xi\in\mathcal{C}^{\infty}(T\overline{M})\phantom{.}\big|\phantom{.}\xi\cdot x\in x^{2}\mathcal{C}^{\infty}(\overline{M})\text{, }\xi_{|\partial\overline{M}}\in\mathcal{C^{\infty}}(\partial M,T\mathcal{F})\right\} ,
\]
and the $\mathcal{F}_{c}$-vector fields, defined as $\mathcal{V}_{\mathcal{F}_{c}}(\overline{M}):=\frac{1}{x}\cdot\mathcal{V}_{\mathcal{F}}(\overline{M})$.
These are smooth sections of special vector bundles over $\overline{M}$
(\cite{[R]}, section 1):
\begin{defn}
The $\mathcal{F}$-tangent bundle $^{\mathcal{F}}T\overline{M}$ is
the unique vector bundle over $\overline{M}$ such that $\mathcal{V}_{\mathcal{F}}(\overline{M})\simeq\mathcal{C}^{\infty}\left(\overline{M},\phantom{}^{\mathcal{F}}T\overline{M}\right)$,
and the $\mathcal{F}_{c}$-tangent bundle $^{\mathcal{F}_{c}}T\overline{M}$
is the one such that $\mathcal{V}_{\mathcal{F}_{c}}(\overline{M})\simeq\mathcal{C}^{\infty}\left(\overline{M},\phantom{}^{\mathcal{F}_{c}}T\overline{M}\right)$.
\end{defn}
Let $c:[0,1[_{x}\times W\rightarrow\overline{M}$ be a local diffeomorphism
onto a tubular neighbourhood of the boundary, and letting $\Gamma$
act trivially on $[0,1[$, define the covering projection 
\[
\hat{\nu}:[0,1[_{x}\times\widetilde{W}\longrightarrow[0,1[_{x}\times W\text{, }(x,p)\longmapsto(x,\nu(p)).
\]
Suppose $\mathcal{U}_{\widetilde{W}}\subset\widetilde{W}$ is an open
subset on which the covering projection $\nu$ restricts to a diffeomorphism
$\mathcal{U}_{\widetilde{W}}\overset{\nu}{\rightarrow}\mathcal{U}_{W}$.
The bundles $^{\mathcal{F}}T\overline{M}$ and $^{\mathcal{F}_{c}}T\overline{M}$
are respectively related to the $\phi$- and the $d$-tangent bundles
defined in section 3 of \cite{[DaW]} by the following isomorphisms:
\[
^{\mathcal{F}}T([0,1[_{x}\times\mathcal{U}_{W})\simeq\hat{\nu}_{*}\left[\phantom{}^{\phi}T([0,1[_{x}\times\mathcal{U}_{\widetilde{W}})\right]\text{, }{}^{\mathcal{F}_{c}}T([0,1[_{x}\times\mathcal{U}_{W})\simeq\hat{\nu}_{*}\left[\phantom{}^{d}T([0,1[_{x}\times\mathcal{U}_{\widetilde{W}})\right].
\]

We will work with the following types of metrics on $\overline{M}$:
\begin{defn}
\label{Def_Foliated_metrics}An \textit{exact foliated boundary metric}
$g_{\mathcal{F}}$ ($\mathcal{F}$-metric for short) on $\overline{M}$
is a smooth Riemannian metric on $^{\mathcal{F}}T\overline{M}$ that
admits a decomposition of the following form on $c([0,1[_{x}\times W)\subset\overline{M}$:
\[
\hat{\nu}^{*}c^{*}g_{\mathcal{F}}=\frac{dx^{2}}{x^{4}}+\frac{\phi^{*}h}{x^{2}}+\kappa+x^{2}B,
\]
where $h\in\Gamma(\Sigma,S^{2}T^{*}\Sigma)$ is a $\Gamma$-invariant
Riemannian metric on the base space of $\widetilde{W}$, $\kappa\in\Gamma(W,S^{2}T^{*}W)$
is a $\Gamma$-invariant tensor restricting to a metric on the fibres
of $\widetilde{W}$ and possiby depending smoothly on $x$, and $B$
is the pullback of a smooth section of $S^{2}[^{\mathcal{F}}T^{*}\overline{M}]$.
If $B\equiv0$, we will say that $g_{\mathcal{F}}$ is an \textit{asymptotic}
$\mathcal{F}$-metric. If we have a decomposition of the form:
\[
\hat{\nu}^{*}c^{*}g_{\mathcal{F}}=\frac{dx^{2}}{x^{4}}+\frac{\phi^{*}h}{x^{2}}+g_{F},
\]
where this time $\{g_{F}(y)\}_{y\in\Sigma}$ is a family of $\Gamma$-invariant
Riemannian metrics on the fibres $F$ of $\widetilde{W}$ smoothly
parametrized by the base $\Sigma$, we say that $g_{\mathcal{F}}$
is a \textit{product-type} $\mathcal{F}$-metric.

A \textit{foliated cusp} metric $g_{\mathcal{F}_{c}}$ on $\overline{M}$
is a Riemannian metric on $^{\mathcal{F}_{c}}T\overline{M}$ that
is obtained from an $\mathcal{F}$-metric by a conformal rescaling
$g_{\mathcal{F}_{c}}=x^{2}g_{\mathcal{F}}$.\end{defn}
\begin{rem*}
$\phantom{\ensuremath{}}$\\
1) The meaning of the term ``product-type'' metric differs from
one author to another. In this paper, we are following the convention
used in \cite{[Vai01]} and \cite{[DaW]}, but in \cite{[MaMel98]}
and \cite{[R]}, ``product-type'' is what we call ``asymptotic''
here.\\
2) Having a globally defined product-type metric requires the choice
of a splitting $T\widetilde{W}=T^{V}\widetilde{W}\oplus T^{H}\widetilde{W}$
into horizontal and vertical subbundles. This is an implicit assumption
in \cite{[DaW]}.\\
3) Following the terminology of \cite{[DaW]}, an asymptotic metric
is defined as a metric of the form $g_{\mathcal{F}}=\tilde{g}_{\mathcal{F}}+xA$,
where $\tilde{g}_{\mathcal{F}}$ is product-type and $A\in\Gamma(\overline{M},S^{2}[^{\mathcal{F}}T^{*}\overline{M}])$
is a tensor such that $A(x^{2}\partial_{x},\cdot)\equiv0$. This is
compatible with our definition in the following sense: in a coordinate
chart near the boundary, we may take the Taylor expansion of the restriction
$(g_{\mathcal{F}})_{|T\mathcal{F}\times T\mathcal{F}}$ w.r.t. the
boundary defining function $x$ to obtain 
\[
\hat{\nu}^{*}c^{*}\left\{ (g_{\mathcal{F}})_{|T\mathcal{F}\times T\mathcal{F}}\right\} =\kappa_{|TF\times TF}=g_{F}+x\cdot\hat{\kappa},
\]
where $\{g_{F}\}$ is a family of metrics on $F$ as in the last definition
(not depending on $x$). We thus have a local decomposition $\kappa=g_{F}+x\cdot\hat{\nu}^{*}c^{*}A$,
with $A\in\Gamma(\overline{M},S^{2}[^{\mathcal{F}}T^{*}\overline{M}])$
satisfying $A(x^{2}\partial_{x},\cdot)\equiv0$.
\end{rem*}
Going back to definition \ref{Def_mfld_foliated_bndry}, we required
the existence of a symmetric bilinear form $\bar{g}:\overline{M}\rightarrow S^{2}T^{*}\overline{M}$
that extends the complete metric $g$ on the interior $M$. The purpose
of $\bar{g}$ is to encode the singular behaviour of $g$ at infinity,
which is why we did not assume that $\bar{g}$ is a smooth section
of $S^{2}T^{*}\overline{M}$. In the next section, we will use the
following definition:
\begin{defn}
A manifold with foliated boundary $(M,g)$ will be called an \textit{$\mathcal{F}$-manifold}
if the extension $\bar{g}$ coincides with an exact $\mathcal{F}$-metric
on the compactification $\overline{M}$. Similarly, if $\bar{g}\equiv g_{\mathcal{F}_{c}}$
for some exact $\mathcal{F}_{c}$-metric, we will call $(M,g)$ an
\textit{$\mathcal{F}_{c}$-manifold}. Moreover, the prefixes ``asymptotic''
and ``product-type'' will be employed to designate the particular
$\mathcal{F}$/$\mathcal{F}_{c}$-metrics considered.
\end{defn}

\section{\textbf{\textup{Results}}}

Throughout this section, $(M,g)$ is a $4k$-dimensional manifold
with foliated geometry at infinity, and for brevity, we use the same
notations as in definitions \ref{Def_mfld_foliated_bndry} and \ref{Def_Foliated_metrics}
without specifying all the properties of the objects $h$, $\kappa$,
$\{g_{F}(y)\}_{y\in\Sigma}$, $A$ and $B$.

\subsection{Index formulae\label{subsec3.1}}

Our first fundamental result is a generalisation of theorems 4.3 and
3.6 of \cite{[DaW]}. As before, $x:\overline{M}\rightarrow\mathbb{R}_{+}$
is a fixed boundary defining function, the map $c:[0,1[_{x}\times W\rightarrow\overline{M}$
is a local diffeomorphism onto a tubular neighborhood of the boundary
$W=\partial\overline{M}$, and $\hat{\nu}:[0,1[_{x}\times\widetilde{W}\rightarrow[0,1[_{x}\times W$
is the covering projection induced by the $\Gamma$-action on the
fibre bundle $F\rightarrow\widetilde{W}\rightarrow^{\phi}(\Sigma,h)$.
\begin{thm}
\label{Index}Let $(M,(g_{\mathcal{F}}){}_{|M})$ be an $\mathcal{F}$-manifold,
where $g_{\mathcal{F}}=\hat{g}_{\mathcal{F}}+x^{2}B$ is an exact
metric on $S^{2}[^{\mathcal{F}}T^{*}\overline{M}]$ such that $\hat{g}_{\mathcal{F}}$
is asymptotic and $B\in\Gamma(S^{2}[^{\mathcal{F}}T^{*}\overline{M}])$.
\\
Let $\hat{g}_{\phi}$ be the pullback of $(\hat{g}_{\mathcal{F}})_{|c([0,1[\times W)}$
to $[0,1[_{x}\times\widetilde{W}$ and $^{\mathcal{F}}\nabla$ be
the Levi-Civita connection of $g_{\mathcal{F}}$. The signature and
Euler characteristic are given by: 
\[
\tau(M)=\int_{M}L(M,\phantom{}^{\mathcal{F}}\nabla)+L_{CS}(\partial\overline{M},\phantom{}^{\mathcal{F}}\nabla)+\frac{1}{|\Gamma|}\left\{ \rho(\widetilde{W},W)-\frac{1}{2}\lim_{\varepsilon\rightarrow0}\eta(\widetilde{W},(\hat{g}_{\phi})_{|x=\varepsilon})\right\} ,
\]
\[
\chi(M)=\int_{M}e(M,\phantom{}^{\mathcal{F}}\nabla)+e_{CS}(\partial\overline{M},\phantom{}^{\mathcal{F}}\nabla),
\]
where the terms $L_{CS}(\partial\overline{M},\phantom{}^{\mathcal{F}}\nabla)\text{ and }e_{CS}(\partial\overline{M},\phantom{}^{\mathcal{F}}\nabla)$
denote the Chern-Simons boundary corrections, \ensuremath{\rho(\widetilde{W},W)}
 is the rho invariant introduced in subsection \ref{subsec2.2}, and
$\eta(\widetilde{W},(g_{\phi})_{|x=\varepsilon})$ is the eta invariant
of the odd signature operator associated to the metric $(g_{\phi})_{|x=\varepsilon}$
on the hypersurface $\{x=\varepsilon\}\subset[0,1[_{x}\times\widetilde{W}$.\\
If $(M,(g_{\mathcal{F}_{c}}){}_{|M}))$ is an $\mathcal{F}_{c}$-manifold
with $g_{\mathcal{F}_{c}}=x^{2}g_{\mathcal{F}}$, one has similar
index formulae as above, with $\phantom{}^{\mathcal{F}_{c}}\nabla$
instead of $\phantom{}^{\mathcal{F}}\nabla$ and $L_{CS}(\partial\overline{M},\phantom{}^{\mathcal{F}_{c}}\nabla)=e_{CS}(\partial\overline{M},\phantom{}^{\mathcal{F}_{c}}\nabla)=0$.
\end{thm}

\begin{rem*}
The signature formula above does not follow from the index formula
of \cite{[R]} since the signature operator is not Fredholm.\end{rem*}
\begin{proof}
Let $0<\varepsilon\ll1$, and set $M_{\varepsilon}=\overline{M}\smallsetminus c([0,\varepsilon[\times W)$.
In this case $\partial M_{\varepsilon}\approx W$ is the hypersurface
$\{x=\varepsilon\}$ in $\overline{M}$, and topologically, one has
$\tau(M_{\varepsilon})=\tau(M)\text{ and }\chi(M_{\varepsilon})=\chi(M)$.
We prove the result for asymptotic metrics $\hat{g}_{\mathcal{F}}$
and $\hat{g}_{\mathcal{F}_{c}}=x^{2}\hat{g}_{\mathcal{F}}$ first,
assuming that they decompose as
\[
\hat{\nu}^{*}c^{*}\hat{g}_{\mathcal{F}}=\frac{dx^{2}}{x^{4}}+\frac{\phi^{*}h}{x^{2}}+\kappa\text{ and }\hat{\nu}^{*}c^{*}\hat{g}_{\mathcal{F}_{c}}=\frac{dx^{2}}{x^{2}}+\phi^{*}h+x^{2}\kappa
\]
on $c([0,1[\times W)\subset\overline{M}$. We recall the following
facts:\\
(a) If $\hat{g}_{\varepsilon}$ is an auxilary metric on $\overline{M}$
such that 
\[
\hat{\nu}^{*}c^{*}\hat{g}_{\varepsilon}=\frac{dx^{2}}{\varepsilon^{4}}+\frac{\phi^{*}h}{\varepsilon^{2}}+\kappa\text{ and }(\hat{g}_{\mathcal{F}})_{|\partial M_{\varepsilon}}\equiv(\hat{g}_{\varepsilon})_{|\partial M_{\varepsilon}},
\]
then it induces the same odd signature operator as $\hat{g}_{\mathcal{F}}$
on $\partial M_{\varepsilon}$.\\
(b) If $P(M_{\varepsilon},\nabla)$ is a characteristic form computed
with the curvature of some Levi-Civita connection $\nabla$, we have
in our case a transgression form $TP(M_{\varepsilon},\phantom{}^{\varepsilon}\hat{\nabla},\phantom{}^{\mathcal{F}}\hat{\nabla})$
satisfying 
\[
dTP(M_{\varepsilon},\phantom{}^{\varepsilon}\hat{\nabla},\phantom{}^{\mathcal{F}}\hat{\nabla})=P(M_{\varepsilon},\phantom{}^{\mathcal{F}}\hat{\nabla})-P(M_{\varepsilon},\phantom{}^{\varepsilon}\hat{\nabla}),
\]
where $\phantom{}^{\mathcal{F}}\hat{\nabla}\text{ and }\phantom{}^{\varepsilon}\hat{\nabla}$
are the Levi-Civita connections of $\hat{g}_{\mathcal{F}}$ and $\hat{g}_{\varepsilon}$
respectively.\\
(c) Going back to the hypotheses on $W\approx\partial M_{\varepsilon}$
(def. \ref{Def_mfld_foliated_bndry}), we can introduce a finite covering
$\Gamma\rightarrow[\varepsilon,1]\times\widetilde{W}\rightarrow[\varepsilon,1]\times\partial M_{\varepsilon}$,
where $\Gamma$ acts trivially on $[\varepsilon,1]$. From subsection
\ref{subsec2.2}, we know that the eta invariants of the odd signature
operators of $(W,(\hat{g}_{\mathcal{F}})_{|x=\varepsilon})$ and $(\widetilde{W},(\hat{g}_{\phi})_{|x=\varepsilon})$
are related by the equation: 
\[
\frac{1}{2}\eta(W,(\hat{g}_{\mathcal{F}})_{|x=\varepsilon})=\frac{1}{|\Gamma|}\left[\frac{1}{2}\eta(\widetilde{W},(\hat{g}_{\phi})_{|x=\varepsilon})-\rho(\widetilde{W},W)\right].
\]
(d) If we look at the metrics $\hat{g}_{\mathcal{F}_{c}}=x^{2}\hat{g}_{\mathcal{F}}$
and $\varepsilon^{2}\hat{g}_{\varepsilon}$ on $M_{\varepsilon}$,
we see that 
\[
\varepsilon^{2}(\hat{g}_{\varepsilon})_{|\partial M_{\varepsilon}}=\varepsilon^{2}(\hat{g}_{\mathcal{F}})_{|\partial M_{\varepsilon}}=(\hat{g}_{\mathcal{F}c})_{|\partial M_{\varepsilon}},
\]
 and since the eta invariant is not modified by rescaling the associated
metric, we moreover have
\[
\eta(\widetilde{W},(\hat{g}_{\varepsilon})_{|x=\varepsilon})=\eta(\widetilde{W},(\hat{g}_{\phi})_{|x=\varepsilon})=\eta(\widetilde{W},(x^{2}\hat{g}_{\phi})_{|x=\varepsilon}).
\]

Applying the Atiyah-Patodi-Singer theorem to $(M_{\varepsilon},\hat{g}_{\varepsilon})$
and using (a) and (b) above, we obtain for any $\varepsilon>0$ sufficiently
small that 
\begin{equation}
\begin{split}\tau(M) & =\int_{M_{\varepsilon}}L(M_{\varepsilon},\phantom{}^{\varepsilon}\hat{\nabla})-\frac{1}{2}\eta(W,(\hat{g}_{\varepsilon})_{|x=\varepsilon})\\
 & =\int_{M_{\varepsilon}}L(M_{\varepsilon},\phantom{}^{\mathcal{F}}\hat{\nabla})-\int_{\partial M_{\varepsilon}}TL(M_{\varepsilon},\phantom{}^{\varepsilon}\hat{\nabla},\phantom{}^{\mathcal{F}}\hat{\nabla})-\frac{1}{2}\eta(W,(\hat{g}_{\varepsilon})_{|x=\varepsilon})\\
 & =\Bigg[\int_{M_{\varepsilon}}L(M_{\varepsilon},\phantom{}^{\mathcal{F}}\hat{\nabla})\Bigg]_{\text{(A)}}-\Bigg[\int_{\partial M_{\varepsilon}}TL(M_{\varepsilon},\phantom{}^{\varepsilon}\hat{\nabla},\phantom{}^{\mathcal{F}}\hat{\nabla})\Bigg]_{\text{(B)}}\\
 & \phantom{=}-\frac{1}{|\Gamma|}\Bigg[\frac{1}{2}\eta(\widetilde{W},(\hat{g}_{\phi})_{|x=\varepsilon})\Bigg]_{\text{(C)}}+\frac{1}{|\Gamma|}\rho(\widetilde{W},W)
\end{split}
\label{Eq_APS_sign_comp}
\end{equation}
and similarly 
\begin{equation}
\chi(M)=\Bigg[\int_{M_{\varepsilon}}e(M_{\varepsilon},\phantom{}^{\mathcal{F}}\hat{\nabla})\Bigg]_{\text{(D)}}-\Bigg[\int_{\partial M_{\varepsilon}}Te(M_{\varepsilon},\phantom{}^{\varepsilon}\hat{\nabla},\phantom{}^{\mathcal{F}}\hat{\nabla})\Bigg]_{\text{(E)}}\label{Eq_APS_Eul_comp}
\end{equation}
It then remains to take the limits of the terms (A)-(E) as $\varepsilon\rightarrow0$.
The terms (A) and (D) tend to the invariant integrals in the statement
of the theorem, the correction $\rho(\widetilde{W},W)$ is invariant
under changes of metrics, while for (B) and (E), we have by definition
that 
\begin{align*}
L_{CS}(\partial\overline{M},\phantom{}^{\mathcal{F}}\hat{\nabla}) & =-\lim_{\varepsilon\rightarrow0}\int_{\partial M_{\varepsilon}}TL(M_{\varepsilon},\phantom{}^{\varepsilon}\hat{\nabla},\phantom{}^{\mathcal{F}}\hat{\nabla}),\\
e_{CS}(\partial\overline{M},\phantom{}^{\mathcal{F}}\hat{\nabla}) & =-\lim_{\varepsilon\rightarrow0}\int_{\partial M_{\varepsilon}}Te(M_{\varepsilon},\phantom{}^{\varepsilon}\hat{\nabla},\phantom{}^{\mathcal{F}}\hat{\nabla}).
\end{align*}
In the case where we use a foliated cusp metric $\hat{g}_{\mathcal{F}_{c}}$,
we obtain the same expressions as above for $\tau(M)$ and $\chi(M)$,
with $^{\mathcal{F}}\nabla$ replaced by $^{\mathcal{F}_{c}}\nabla$,
and $\eta(\widetilde{W},(\hat{g}_{\phi})_{|x=\varepsilon})$ replaced
by $\eta(\widetilde{W},(x^{2}\hat{g}_{\phi})_{|x=\varepsilon})$.
By facts (c) and (d), these eta invariants coincide. Regarding the
Chern-Simons corrections, since $\nu:\widetilde{W}\rightarrow W$
is a local isometry, we have: : 
\begin{align*}
\int_{\partial M_{\varepsilon}}TL(M_{\varepsilon},\phantom{}^{\varepsilon}\hat{\nabla},\phantom{}^{\mathcal{F}_{c}}\hat{\nabla}) & =\frac{1}{|\Gamma|}\int_{\{\varepsilon\}\times\widetilde{W}}TL([\varepsilon,1[\times\widetilde{W},\hat{\nu}^{*}(\phantom{}^{\varepsilon}\hat{\nabla}),\phantom{}^{d}\hat{\nabla})_{|x=\varepsilon},\\
\int_{\partial M_{\varepsilon}}Te(M_{\varepsilon},\phantom{}^{\varepsilon}\hat{\nabla},\phantom{}^{\mathcal{F}_{c}}\hat{\nabla}) & =\frac{1}{|\Gamma|}\int_{\{\varepsilon\}\times\widetilde{W}}Te([\varepsilon,1[\times\widetilde{W},\hat{\nu}^{*}(\phantom{}^{\varepsilon}\hat{\nabla}),\phantom{}^{d}\hat{\nabla})_{|x=\varepsilon},
\end{align*}
where $^{d}\hat{\nabla}$ the Levi-Civita connection of $x^{2}\hat{g}_{\phi}$.
Taking $\varepsilon\rightarrow0$ then yields $L_{CS}(\partial\overline{M},\phantom{}^{\mathcal{F}_{c}}\hat{\nabla})=e_{CS}(\partial\overline{M},\phantom{}^{\mathcal{F}_{c}}\hat{\nabla})=0$
by proposition \ref{Prop_CS_prod_exact}.

If we now consider an exact metric $g_{\mathcal{F}}=\hat{g}_{\mathcal{F}}+x^{2}B$
with Levi-Civita connection $^{\mathcal{F}}\nabla$, the transgression
integrals in equations (\ref{Eq_APS_sign_comp}) and (\ref{Eq_APS_Eul_comp})
are replaced by $\int_{\partial M_{\varepsilon}}TL(M_{\varepsilon},\phantom{}^{\varepsilon}\hat{\nabla},\phantom{}^{\mathcal{F}}\nabla)$
and $\int_{\partial M_{\varepsilon}}Te(M_{\varepsilon},\phantom{}^{\varepsilon}\hat{\nabla},\phantom{}^{\mathcal{F}}\nabla)$
respectively. Again, by expressing these as integrals on $\{\varepsilon\}\times\widetilde{W}$
and applying proposition \ref{Prop_CS_prod_exact}, we have: 
\begin{align*}
L_{CS}(\partial\overline{M},\phantom{}^{\mathcal{F}}\nabla) & =-\lim_{\varepsilon\rightarrow0}\int_{\partial M_{\varepsilon}}TL(M_{\varepsilon},\phantom{}^{\varepsilon}\hat{\nabla},\phantom{}^{\mathcal{F}}\nabla)=-\lim_{\varepsilon\rightarrow0}\int_{\partial M_{\varepsilon}}TL(M_{\varepsilon},\phantom{}^{\varepsilon}\hat{\nabla},\phantom{}^{\mathcal{F}}\hat{\nabla}),\\
e_{CS}(\partial\overline{M},\phantom{}^{\mathcal{F}}\nabla) & =-\lim_{\varepsilon\rightarrow0}\int_{\partial M_{\varepsilon}}Te(M_{\varepsilon},\phantom{}^{\varepsilon}\hat{\nabla},\phantom{}^{\mathcal{F}}\nabla)=-\lim_{\varepsilon\rightarrow0}\int_{\partial M_{\varepsilon}}Te(M_{\varepsilon},\phantom{}^{\varepsilon}\hat{\nabla},\phantom{}^{\mathcal{F}}\hat{\nabla}),
\end{align*}
and likewise, $L_{CS}(\partial\overline{M},\phantom{}^{\mathcal{F}_{c}}\nabla)=e_{CS}(\partial\overline{M},\phantom{}^{\mathcal{F}_{c}}\nabla)=0$
for exact foliated cusp metrics.\end{proof}
\begin{rem*}
The reader should be warned that the corresponding result in \cite{[DaW]}
contains a mistake. In theorems 3.6 and 4.3 of \cite{[DaW]}, it is
stated that the Chern-Simons terms vanish for fibres $F$ of arbitrary
(positive) dimension, but this is false in general. A counter example
for the Euler characteristic for even dimensional fibres is given
by $M=\mathbb{R}^{2k}\times S^{2p}$ equipped with the product of
the standard metrics $g$, which is a fibred-boundary metric for the
boundary defining function $1/r$, where $r$ is the radial coordinate
on $\mathbb{R}^{2k}$. The boundary at infinity is $S^{2k-1}\times S^{2p}$,
viewed as a trivial $S^{2p}$-bundle over $S^{2k-1}$. The Pfaffian
of the curvature of $g$ vanishes in this case since one of the factors
is flat, but we still have $\chi(M)\neq0$, which means that the Chern-Simons
term does not vanish. Dai and Wei have been made aware of this mistake,
and they intend to publish an erratum that rectifies this issue (see
last subsection of appendix \ref{Apndx1} for more details).
\end{rem*}
As seen in the previous proof, the fact that $\nu:\widetilde{W}\rightarrow W$
is a local isometry allows us to compute transgression integrals on
$[0,1[_{x}\times\widetilde{W}$ endowed with a fibred boundary metric.
As an obvious consequence of proposition \ref{Prop_CS_vanishing},
we have the following sufficient conditions for the vanishing of the
Chern-Simons corrections with an $\mathcal{F}$-metric:
\begin{prop}
\label{Prop_CS_vanish_foliate} Let $\overline{M}$ be a $4k$-dimensional
manifold with foliated boundary. Suppose there's a splitting $T\widetilde{W}=T^{V}\widetilde{W}\oplus T^{H}\widetilde{W}$
into horizontal and vertical subbundles for the cover $\widetilde{W}\rightarrow_{\phi}\Sigma$
of $W=\partial\overline{M}$, and consider an exact $\mathcal{F}$-metric
$g_{\mathcal{F}}$ on $^{\mathcal{F}}T\overline{M}$ of the form
\[
g_{\mathcal{F}}=\tilde{g}_{\mathcal{F}}+x\cdot A+x^{2}\cdot B\text{, with }A,B\in\Gamma(\overline{M},S^{2}[^{\mathcal{F}}T^{*}\overline{M}]),
\]
where $\tilde{g}_{\mathcal{F}}$ is a product metric, and where the
pullback $\widetilde{A}=\hat{\nu}^{*}c^{*}(A_{|c([0,1[\times W)})$
of $A$ satisfies the following: 
\begin{align}
\text{(i) } & \widetilde{A}(x^{2}\partial_{x},\cdot)\equiv0\text{ and }\widetilde{A}(xY_{1},xY_{2})=O(x),\nonumber \\
\text{(ii) } & \left(V_{1}\cdot\widetilde{A}(V_{2},xY_{1})-V_{2}\cdot\widetilde{A}(V_{1},xY_{1})-\widetilde{A}([V_{1},V_{2}],xY_{1})\right)_{|x=0}=0,\label{Eq_bndry_cond_foliated}
\end{align}
for all $Y_{1},Y_{2}\in\mathfrak{X}(\Sigma)$ and all $V_{1},V_{2}\in\Gamma(\widetilde{W},T^{V}\widetilde{W})$.
Then:\\
(a) If the dimension of the fibres $F$ of $\widetilde{W}$ is odd,
then the boundary correction to the Euler characteristic vanishes:
\[
e_{CS}(\partial\overline{M},\phantom{}^{\mathcal{F}}\nabla)=0.
\]
(b) If $\dim F=1$, then the correction to the Hirzebruch signature
vanishes:
\[
L_{CS}(\partial\overline{M},\phantom{}^{\mathcal{F}}\nabla)=0.
\]

\end{prop}
An interesting case for applications to asymptotically locally flat
(ALF) manifolds, is when $\text{dim}(M)=4$ and $\phi:\widetilde{W}\rightarrow\Sigma$
is a circle bundle. We now list the assumptions made for the remaining
of this section:

\textit{$\bullet$ The circle bundle $\widetilde{W}\rightarrow\Sigma$:}
Suppose $\Sigma$ is a compact Riemann surface, and let $(E,g_{E})\rightarrow(\Sigma,h)$
be an oriented rank 2 Euclidian vector bundle. We denote by $\widetilde{W}=S(E)$
the circle subbundle of $E$, and by $\widetilde{X}=D(E)$ the disc
bundle over $\Sigma$ such that $\partial\widetilde{X}=\widetilde{W}$.\\
\textit{$\bullet$ Action of $\Gamma$:} The finite group $\Gamma$
acts as follows on $\widetilde{W}=S(E)$: On the surface $\Sigma$,
the action is smooth and there are only isolated fixed points.  We
assume that we have a faithful representation of $\Gamma$ on $\mathbb{C}$,
and that the action on $S^{1}$ is smooth and free. This $\Gamma$-action
is extended to $\widetilde{X}=D(E)$ as follows: the action remains
the same on the base $\Sigma$, and if $a\in\Gamma$ acts on a fibre
of $\widetilde{W}$ via $\theta\mapsto a(\theta)$, then the same
isometry sends a point $(r,\theta)$ in a fibre of $\widetilde{X}$
to $(r,a(\theta))$. With this requirement, the fixed points are isolated,
and strictly included in $\widetilde{X}\smallsetminus\widetilde{W}$.\\
$\bullet$ \textit{Foliated boundary metrics on $\overline{M}$:}
The Levi-Civita connection of $g_{E}$ induces a splitting $T\widetilde{W}=T^{V}\widetilde{W}\oplus T^{H}\widetilde{W}$,
so we may consider product-type foliated boundary metrics on $\overline{M}$.
We assume that $\mathcal{F}$-metrics on $\overline{M}$ decompose
as
\begin{equation}
\hat{\nu}^{*}c^{*}g_{\mathcal{F}}=\left(\frac{dx^{2}}{x^{4}}+\frac{\phi^{*}h}{x^{2}}+g_{F}\right)+\hat{\nu}^{*}c^{*}\left(xA+x^{2}B\right)\label{Eq_exact_F}
\end{equation}
near $W=\partial\overline{M}$, where $g_{F}\in\Gamma(\widetilde{W},S^{2}[T^{V}\widetilde{W}]^{*})$
is a family of Riemannian metrics on $S^{1}$ smoothly parametrized
by $\Sigma$, and $A,B\in\Gamma(\overline{M},S^{2}[^{\mathcal{F}}T^{*}\overline{M}])$
are symmetric bilinear forms such that the pullback $\widetilde{A}$
of $A_{|c([0,1[\times W)}$ to $[0,1[_{x}\times\widetilde{W}$ satisfies
condition (i) in proposition \ref{Prop_CS_vanish_foliate} ((ii) is
automatically satisfied since the fibre is $1$-dimensional). The
only difference with the exact metric definition in \ref{Def_Foliated_metrics}
is the decomposition $\kappa=g_{F}+x\widetilde{A}$.

Under these assumptions, and using the notations of subsection \ref{subsec2.1},
one has: 
\begin{cor}
\label{Zsign}Let $M$ be a $4$-dimensional manifold with foliated
geometry at infinity satisfying the hypotheses above, and let $g$
be an exact $\mathcal{F}$- or $\mathcal{F}_{c}$-metric on $\overline{M}$.
The topological invariants of $M$ are given by: 
\begin{align*}
\tau(M) & =\int_{M}L(M,\phantom{}^{g}\nabla)-\frac{1}{|\Gamma|}\left\{ {\displaystyle \sum_{a\neq Id}\sum_{z\in\text{Fix}(a)}def(a,\phantom{}^{g}\widetilde{B})[z]+\frac{\chi(E)}{3}}\right\} +\epsilon(E),
\end{align*}
\[
\chi(M)=\int_{M}e(M,\phantom{}^{g}\nabla),
\]
where $\chi(E)$ is the Euler characteristic of the vector bundle
$E$, and \textup{$\epsilon(E)$} is defined as
\[
\epsilon(E)=\left\{ \begin{array}{ll}
-1; & \chi(E)<0\\
\phantom{+}0; & \chi(E)=0\\
+1; & \chi(E)>0
\end{array}\right.
\]

\end{cor}
This is an immediate consequence of the following lemma:
\begin{lem}
\label{EtaRhoLmma}For an $\mathcal{F}$-metric of the form (\ref{Eq_exact_F})
satisfying equations (\ref{Eq_bndry_cond_foliated}), the Chern-Simons
terms vanish:
\[
L_{CS}(\partial\overline{M},\phantom{}^{\mathcal{F}}\nabla)=e_{CS}(\partial\overline{M},\phantom{}^{\mathcal{F}}\nabla)=0,
\]
the limit of the eta invariant in theorem \ref{Index} coincides with
the following adiabatic limit:
\[
\frac{1}{2}\lim_{\varepsilon\rightarrow0}\eta(\widetilde{W},(g_{\phi})_{|x=\varepsilon})=\frac{1}{2}a{\displaystyle {\lim_{\varepsilon\rightarrow0}}\eta(S(E),\varepsilon^{-2}\phi^{*}h+g_{F})=\frac{\chi(E)}{3}-\epsilon(E)},
\]
and the rho invariant is given by:
\[
\rho(\widetilde{W},W)=(|\Gamma|-1)\epsilon(E)-{\displaystyle {\sum_{a\neq Id}\sum_{z\in\text{Fix}(a)}}def(a,\phantom{}^{g}\widetilde{B})[z]}.
\]
\end{lem}
\begin{proof}
By proposition \ref{Prop_CS_vanish_foliate}, it is sufficient to
consider that the $\mathcal{F}$-metric $g_{\mathcal{F}}$ and the
auxiliary metric $g_{\varepsilon}$ are of product-type:
\[
\hat{\nu}^{*}c^{*}g_{\mathcal{F}}=\left(\frac{dx^{2}}{x^{4}}+\frac{\phi^{*}h}{x^{2}}+g_{F}\right)\text{ and }\hat{\nu}^{*}c^{*}g_{\varepsilon}=\left(\frac{dx^{2}}{\varepsilon^{4}}+\frac{\phi^{*}h}{\varepsilon^{2}}+g_{F}\right).
\]
This is because the perturbations $xA+x^{2}B\in\Gamma(S^{2}[^{\mathcal{F}}T^{*}\overline{M}])$
to $g_{\mathcal{F}}$ that we are allowing yield vanishing transgression
integrals when $\varepsilon\rightarrow0$. Also, since $(g_{\phi})_{|x=\varepsilon}\equiv(\varepsilon^{-2}\phi^{*}h+g_{F})$
on $[0,1[_{x}\times\widetilde{W}$, theorem \ref{Index} and proposition
\ref{Prop_CS_vanish_foliate} imply that: 
\[
\chi(M)=\int_{M}e(M,\phantom{}^{\mathcal{F}}\nabla)\text{ and }\tau(M)=\int_{M}L(M,\phantom{}^{\mathcal{F}}\nabla)-\frac{1}{2}\lim_{\varepsilon\rightarrow0}\eta(\widetilde{W},\varepsilon^{-2}\phi^{*}h+g_{F}).
\]
The last limit coincides with the adiabatic limit of the eta invariant
of the odd signature operator associated to the metric $(\varepsilon^{-2}\phi^{*}h+g_{F})$
on $\widetilde{W}=S(E)$, so as a special case of theorem 3.2 in \cite{[DaZ]}
(also formula (5.4) in \cite{[DaW]} ), we have:
\[
\frac{1}{2}a{\displaystyle {\lim_{\varepsilon\rightarrow0}}\eta(S(E),\varepsilon^{-2}\phi^{*}h+g_{F})=\frac{\chi(E)}{3}-\epsilon(E)}.
\]
To determine the expression of the rho invariant, we use the results
(and notations) of subsection \ref{subsec2.1} on the spaces \ensuremath{\widetilde{X}=D(E)}
 and $\widetilde{W}=S(E)$. From the $G$-signature formula and the
definition of \ensuremath{\rho}
 in terms of the $G$-eta invariants, we have: 
\[
\rho(\widetilde{W},W)={\displaystyle {\sum_{a\neq Id}}\left[\tau(a,\widetilde{X})-{\displaystyle {\sum_{z\in\text{Fix}(a)}}def(a,\phantom{}^{g}\widetilde{B})[z]}\right]}.
\]
It remains to specify the value of $\tau(a,\widetilde{X})=\text{Tr}\left(a^{*}|_{\widehat{H}_{+}^{2}}\right)-\text{Tr}\left(a^{*}|_{\widehat{H}_{-}^{2}}\right)$
under our hypotheses. We have the following (complex) cohomology group
isomorphisms 
\[
H^{2}(\widetilde{X})\overset{\text{homotopy}}{\simeq}H^{2}(\Sigma)\simeq\mathbb{C},
\]
\[
H^{2}(\widetilde{X},\widetilde{W})\overset{\text{def.}}{\simeq}H_{c}^{2}(\widetilde{X}\smallsetminus\widetilde{W})\overset{\text{Poincaré}}{\simeq}H^{2}(\widetilde{X}\smallsetminus\widetilde{W})\overset{\text{homotopy}}{\simeq}H^{2}(\Sigma)\simeq\mathbb{C}.
\]
Since the $a^{*}$ are induced by orientation preserving isometries,
we have that $a^{*}|_{H^{2}(\widetilde{X})}=Id_{H^{2}(\widetilde{X})}$.
Moreover, the image $\widehat{H}^{2}=\widehat{H}_{+}^{2}\oplus\widehat{H}_{-}^{2}$
is at most 1-dimensional, so one of the subspaces must be trivial,
hence $\tau(a,\widetilde{X})\in\{\pm1,0\}$. On the other hand, the
spaces $\widehat{H}_{\pm}^{2}$ are the subspaces on which the quadratic
form given by 
\[
Q:H^{0}(\Sigma)\otimes H^{0}(\Sigma)\rightarrow\mathbb{R};\alpha\otimes\beta\mapsto\big<\alpha\smile\beta\smile\left[e_{E}\right],\left[\Sigma\right]\big>,
\]
is either positive or negative definite ($\left[e_{E}\right]$ is
the Euler class of $E$ here, c.f. section 5 of \cite{[DaW]}). This
results from the Thom isomorphism $\cdot\smile\Phi:H^{0}(\Sigma)\rightarrow H_{c}^{2}(\widetilde{X}\smallsetminus\widetilde{W})$
and the fact that $\big<\Phi^{2},[\widetilde{X}]\big>=\big<\left[e_{E}\right],\left[\Sigma\right]\big>$,
where $\Phi$ denotes the Thom class. The sign of $\tau(a,\widetilde{X})$
is hence that of $Q(1,1)=\chi(E)$, i.e $\tau(a,\widetilde{X})=\epsilon(E)$,
and the result follows.\end{proof}
\begin{rem*}
For practical computations, having explicit expressions for the action
of $\Gamma$ should allow one to determine coherent systems of angles
$\{\theta_{a,j}(z)\}_{j=1}^{2}$ for the fixed points $z\in\text{Fix}(a)\text{ of }a\in\Gamma$.
In this case, the boundary correction to the signature is 
\[
\frac{1}{2}\lim_{\varepsilon\rightarrow0}\eta(W,(g_{\mathcal{F}})_{|x=\varepsilon})=\frac{1}{|\Gamma|}\left[\frac{\chi(E)}{3}-\sum_{a\neq Id}\sum_{z\in\text{Fix}(a)}\prod_{j=1}^{2}\cot\left(\theta_{a,j}(z)/2\right)\right]-\epsilon(E).
\]

\end{rem*}

\subsection{Hitchin-Thorpe inequality}

Here, we are interested in an obstruction to the existence of Einstein
metrics on the noncompact $4$-manifolds that we are considering.
\begin{thm}
\label{HTIFB}Let \ensuremath{M}
 be a $4$-dimensional manifold with foliated geometry at infinity
satisfying the assumptions listed before corollary \ref{Zsign}. If
$M$ admits an exact Einstein $\mathcal{F}$- or $\mathcal{F}_{c}$-metric,
then 
\[
\chi(M)\geq\frac{3}{2}\Bigg|\tau(M)-\epsilon(E)+\frac{1}{|\Gamma|}\left\{ {\displaystyle {\sum_{a\neq Id}\sum_{z\in\text{Fix}(a)}def(a,\phantom{}^{g}\widetilde{B})[z]}+\frac{\chi(E)}{3}}\right\} \Bigg|.
\]
If equality occurs, then the universal cover of $M$ is a complete
Ricci-flat (anti-)self-dual manifold.\end{thm}
\begin{proof}
On a 4-dimensional Riemannian manifold $(\overline{M},g)$ (with $g$
an exact foliated boundary or foliated cusp metric), the Riemann tensor
decomposes as (1.128, \cite{[B2]}) 
\[
^{g}R=\begin{pmatrix}W^{+}+\frac{S}{12}\text{Id} & Z\\
Z^{t} & W^{-}+\frac{S}{12}\text{Id}
\end{pmatrix},
\]
where $W^{+}\text{ and }W^{-}$ are respectively the self-dual and
anti-self-dual parts of the Weyl tensor, $S$ is the scalar curvature
and $Z=Ric-\frac{S}{4}\text{Id}$ is the traceless part of the Ricci
tensor. This decomposition allows us to re-express the Euler form
and the Hirzebruch $L$-polynomial in terms of the components of $^{g}R$,
so that our invariant integrals become (6.31 and 6.34, \cite{[B2]}):
\begin{align*}
\frac{3}{2}\int_{M}L(M,\phantom{}^{g}\nabla) & =\frac{1}{8\pi^{2}}\int_{M}\big(|W^{+}|^{2}-|W^{-}|^{2}\big)dVol^{g}\\
 & =\frac{3}{2}\left[\tau(M)+\frac{1}{|\Gamma|}\Big[\frac{1}{2}a{\displaystyle {\lim_{\varepsilon\rightarrow0}}\eta(\widetilde{A}_{\varepsilon})-\rho(\widetilde{W},W)\Big]}\right],\\
\int_{M}e(M,\phantom{}^{g}\nabla) & =\frac{1}{8\pi^{2}}\int_{M}\left(|W^{+}|^{2}+|W^{-}|^{2}-|Z|^{2}+\frac{S^{2}}{24}\right)dVol^{g}=\chi(M).
\end{align*}
Using the fact that $Z=0$ if $\ensuremath{g}$ is an Einstein metric,
adding and subtracting the integrals above leads to 
\[
\chi(M)\geq\frac{3}{2}\Bigg|\tau(M)+\frac{1}{|\Gamma|}\Big[\frac{1}{2}a{\displaystyle {\lim_{\varepsilon\rightarrow0}}\eta(\widetilde{A}_{\varepsilon})-\rho(\widetilde{W},W)\Big]}\Bigg|,
\]
and the inequality of the statement follows from lemma \textit{\ref{EtaRhoLmma}}.
When equality occurs, the equations above yield 
\[
\int_{M}\left(|W^{\pm}|^{2}+\frac{S^{2}}{24}\right)dVol^{g}=0\Longrightarrow S=|W^{\pm}|=0.
\]
The vanishing of the scalar curvature means that $M$ is Ricci-flat,
and the vanishing of the (anti-) self-dual part of the Weyl tensor
implies that the holonomy group of the pullback of $g$ to the universal
cover of $M$ is contained in $SU(2)$ (by section 3 of \cite{[H]}),
which proves the second claim of the theorem.
\end{proof}

\section{\textbf{Examples}}

Our examples are built on quotients of the Gibbons-Hawking ansatz
by cyclic groups, which are studied extensively in \cite{[GH]}. In
appendix \ref{Apndx2}, we give some details and references on multi-Taub-NUT
metrics. Let $k\geq1$ be an integer, and suppose that $\{p_{j}\}_{j=1}^{k}$
are points in $\mathbb{R}^{3}$ of coordinates $p_{j}\equiv(\cos(2\pi j/k),\sin(2\pi j/k),0)$
with respect to the origin. We consider the principal circle bundle
\[
S^{1}\longrightarrow M\overset{\pi}{\longrightarrow}\mathbb{R}^{3}\smallsetminus\{p_{j}\}_{j=1}^{k},
\]
whose first Chern class yields $(-1)$ when paired with the homology
class associated to a sphere centred at one of the monopoles $p_{j}$.
We equip $M$ with the multi-Taub-NUT metric 
\[
g=\pi^{*}[V\cdot((dx^{1})^{2}+(dx^{2})^{2}+(dx^{3})^{2})]+\pi^{*}(V^{-1})\cdot(d\theta+\pi^{*}\omega)^{2},
\]
where the function $V:\mathbb{R}^{3}\smallsetminus\{p_{j}\}_{j=1}^{k}\rightarrow\mathbb{R}$
is defined as: 
\[
V(x)=1+\frac{1}{2}{\displaystyle {\sum_{j=1}^{k}}\frac{1}{|x-p_{j}|}},
\]
and determines $g$ uniquely, $\theta$ is a coordinate on the fibres,
and $\omega$ is a connection 1-form on $\mathbb{R}^{3}\smallsetminus\{p_{j}\}$
(unique up to ``gauge transformations'') of curvature $d\omega=*_{\mathbb{R}^{3}}dV$.
By suitably adding points $\{q_{i}\}$ to $M$ above the monopoles,
we obtain a smooth Ricci-flat hyper-K\"ahler completion $(M_{0},g_{0})$
of $(M,g)$ \cite{[AKL]}, with only one asymptotically locally flat
end at infinity. If $(\overline{M}_{0},\bar{g}_{0})$ is the compactification
at infinity of $(M_{0},g_{0})$, the boundary $\partial\overline{M}_{0}$
is diffeomorphic to a quotient of a Hopf fibration by a $\mathbb{Z}_{k}$-action
on the fibres \cite{[H79]}, i.e. it is the total space of a fibration
of the form
\[
S^{1}\longrightarrow S^{3}/\mathbb{Z}_{k}\overset{\phi}{\longrightarrow}S^{2}.
\]

Let $(r,\varphi,\psi)$ be the spherical coordinates on $\mathbb{R}^{3}$
($r=\sqrt{(x^{1})^{2}+(x^{2})^{2}+(x^{3})^{2}}$), and consider the
metric $h=(d\varphi)^{2}+\sin^{2}\varphi(d\psi)^{2}$ on the base
$S^{2}$ of $\partial\overline{M}_{0}$. If we introduce the boundary
defining function $\rho:=1/\pi^{*}(r)$ on $\pi^{-1}(\{r>1\})\subset\overline{M}_{0}$,
then the multi-Taub-NUT metric can be rewritten as:
\[
g_{0}=\frac{d\rho^{2}}{\rho^{4}}+\frac{\phi^{*}h}{\rho^{2}}+\kappa,
\]
where $\kappa=V^{-1}(d\theta+\omega)^{2}\in\Gamma(\partial\overline{M}_{0},S^{2}T^{*}\partial\overline{M}_{0})$
restricts to the metric $V^{-1}(d\theta)^{2}$ on the fibres, and
depends smoothly on $\rho$, which means that $g_{0}$ is an asymptotic
$\phi$-metric.

In the examples below, we consider quotients of $\partial\overline{M}_{0}$
by the action of a cyclic group, in which case $\overline{g}_{0}$
gives rise to an asymptotic $\mathcal{F}$-metric (after averaging
with respect to the group action).

Finally, the invariants of interest for our discussion are the following
(\cite{[GH]}, (5.4) in \cite{[DaW]}): 
\[
\chi(M_{0})=k\text{, }\tau(M_{0})=1-k\text{ and }\frac{1}{2}a\lim\eta\left(\partial\overline{M}_{0},\bar{g}_{0|\partial\overline{M}_{0}}\right)=\frac{k}{3}-1.
\]
To simplify the notations, we will write $\rho$ instead of $\rho(\partial\overline{M}_{0},\partial\overline{M}_{0}/\Gamma)$,
and $\eta^{A}$ instead of $\frac{1}{2}a\lim\eta(\partial\overline{M}_{0},\bar{g}_{0|\partial\overline{M}_{0}})$.
\begin{example}
Let $\Gamma=\mathbb{Z}_{k}$, and suppose that the generator $1\in\Gamma$
acts as follows on $\overline{M}_{0}$: On the base $\mathbb{R}^{3}$,
the group $\Gamma$ acts via a rotation of \ensuremath{+\frac{2\pi}{k}}
 about the axis \textit{Oz} (w.r.t the canonical orientation). As
explained on p.106 of \cite{[W]}, this action can be lifted to an
action on $\overline{M}_{0}$ by isometries, but not in a unique fashion.
The action on $\overline{M}_{0}$ is uniquely determined once we specify
how $\Gamma$ acts on the fibre above the origin in $\mathbb{R}^{3}$,
and to ensure that it is free, we require that $1\in\Gamma$ acts
as multiplication by $e^{i\frac{2\pi}{k}}$ on $\pi^{-1}(\{0\})\approx S^{1}$
(as it is done in section 6 of \cite{[GRR]}). We consider the manifold
$M_{1}:=\overline{M}_{0}/\Gamma$ with boundary $\partial M_{1}=\partial\overline{M}_{0}/\Gamma$,
and induced metric $g_{1}$, which is a Ricci-flat foliated boundary
metric (thus Einstein in particular). Since the quotient map $\overline{M}_{0}\rightarrow M_{1}$
is a local isometry, we have: 
\[
\begin{split}\chi(M_{1})-\frac{3}{2}\Big|\tau(M_{1})+\frac{1}{k}\big(\eta^{A}-\rho\big)\Big| & =\int_{M_{1}}e(M_{1},\phantom{}^{g_{1}}\nabla)-\frac{3}{2}\Big|\int_{M_{1}}L(M_{1},\phantom{}^{g_{1}}\nabla)\Big|\\
 & =\frac{1}{k}\Bigg\{\int_{\overline{M}_{0}}e(\overline{M}_{0},\phantom{}^{g_{0}}\nabla)-\frac{3}{2}\Big|\int_{\overline{M}_{0}}L(\overline{M}_{0},\phantom{}^{g_{0}}\nabla)\Big|\Bigg\}\\
 & =\frac{1}{k}\Big\{\chi(\overline{M}_{0})-\frac{3}{2}\Big|\tau(\overline{M}_{0})+\tilde{\eta}\Big|\Big\}\\
 & =\frac{1}{k}\Big\{ k-\frac{3}{2}\big|1-k+\frac{k}{3}-1\big|\Big\}=\frac{k-k}{k}=0.
\end{split}
\]
We note that our inequality reduces to that of theorem 1.1 (and cor.
1.2) of \cite{[DaW]} since $M_{1}$ is a global quotient and $\overline{M}_{0}$
is a manifold with fibred boundary.
\end{example}

\begin{example}
We consider $M_{1}$ of the previous example, and we make $l\geq1$
blow-ups in its interior. We obtain the space $X_{1}=M_{1}\#l\overline{\mathbb{C}P}^{2}$,
where $\overline{\mathbb{C}P}^{2}$ is the complex projective plane
with reverse orientation, for which we have $\chi(\overline{\mathbb{C}P}^{2})=3$
and $\tau(\overline{\mathbb{C}P}^{2})=-1$. By the identities
\begin{align*}
\tau(A\#B) & =\tau(A)+\tau(B),\\
\chi(A\#B) & =\begin{cases}
\chi(A)+\chi(B) & \text{for }\dim A=\dim B=2m+1\\
\chi(A)+\chi(B)-2 & \text{for }\dim A=\dim B=2m
\end{cases},
\end{align*}
we have 
\begin{align*}
\tau(X_{1})+\frac{1}{k}(\eta^{A}-\rho) & =\frac{1}{k}\left[\int_{\overline{M}_{0}}L(\overline{M}_{0},\phantom{}^{g_{0}}\nabla)\right]+l\cdot\tau(\overline{\mathbb{C}P}^{2})=-\Big(\frac{2}{3}+l\Big),\\
\chi(X_{1}) & =\frac{1}{k}\left[\int_{\overline{M}_{0}}e(\overline{M}_{0},\phantom{}^{g_{0}}\nabla)\right]+l\cdot\left(\chi(\overline{\mathbb{C}P}^{2})-2\right)=1+l.
\end{align*}
Thus, for any $l\geq1$, we get 
\[
\chi(X_{1})-\frac{3}{2}\Bigg|\tau(X_{1})+\frac{1}{k}(\eta^{A}-\rho)\Bigg|=(l+1)-\Big|-\Big(1+\frac{3}{2}l\Big)\Big|=-\frac{l}{2}<0.
\]
Using the expressions of the invariant integrals in terms of the tensors
$W^{\pm}\text{, }Z$ and $\frac{S}{4}\text{Id}$ (as in the proof
of \ref{HTIFB}), this inequality implies that $|Z|^{2}>2|W^{\pm}|^{2}+(S^{2}/24)\geq0$,
which means that $X_{1}$ doesn't admit an Einstein exact $\mathcal{F}/\mathcal{F}_{c}$-metric
since the component $Z=Ric-\frac{S}{4}\text{Id}$ can't vanish.
\end{example}

\begin{example}
We start by modifying $M_{1}$ of the previous examples to construct
a space which isn't a global quotient of a manifold with fibred boundary.
The presentation here is largely based upon example 26 of \cite{[GRR]}.
Let $B\Gamma$ be the classifying space of $\Gamma$-bundles, and
$\theta:\partial M_{1}\rightarrow B\Gamma$ the map corresponding
to the covering $\partial\overline{M}_{0}\rightarrow\partial M_{1}=\partial\overline{M}_{0}/\Gamma$.
Let $\gamma:S^{1}\rightarrow\partial M_{1}$ be a loop such that $[\theta\circ\gamma]\in\pi_{1}(B\Gamma)$
is nonzero, and $\tilde{\gamma}:S^{1}\rightarrow M_{1}\smallsetminus\partial M_{1}$
a smooth translation of $\gamma$ into the interior of $M_{1}$. The
loop $\tilde{\gamma}(S^{1})$ has a trivialized tubular neighborhood
in $M_{1}$ since the tangent bundle of $\partial M_{1}$ is trivial
(the boundary is 3-dimensional and orientable). We now perform a codimension
3 surgery along $\tilde{\gamma}(S^{1})\subset M_{1}$ (\cite{[LM]},
p.299): let $\mathcal{V}\approx S^{1}\times\overline{\mathbb{B}}^{3}$
be a closed subset of $M_{1}$ such that $\tilde{\gamma}(S^{1})\subset\partial\mathcal{V}$,
and consider $N\approx\overline{\mathbb{D}}^{2}\times S^{2}$ such
that $\partial\mathcal{V}\approx\partial N\approx S^{1}\times S^{2}$,
then define 
\[
M_{2}:=(M_{1}\smallsetminus\mathcal{V}^{\circ})\cup_{\partial N}N,
\]
with $\partial M_{2}=\partial M_{1}$. In the manifold $M_{2}$, the
loop $\tilde{\gamma}(S^{1})$ is contractible, which implies that
the map $\theta:\partial M_{2}\rightarrow B\Gamma$ cannot be extended
to the interior $M_{2}\smallsetminus\partial M_{2}$, and hence that
$M_{2}$ is not a global quotient.

We need to determine the topological invariants of $M_{2}$. Since
$\partial M_{2}=\partial\overline{M}_{0}/\Gamma$, we have the same
boundary correction as before. For the Euler characteristic, we use
the identity 
\[
\chi(A\cup B)=\chi(A)+\chi(B)-\chi(A\cap B)
\]
on $M_{2}=(M_{2}\smallsetminus N^{\circ})\cup N$ and $M_{1}=(M_{1}\smallsetminus\mathcal{V}^{\circ})\cup\mathcal{V}$
to obtain 
\[
\chi(M_{2})=\chi(M_{1})+2=2+\frac{1}{k}\left[\int_{\overline{M}_{0}}e(\overline{M}_{0},\phantom{}^{g_{0}}\nabla)\right]=3,
\]
since $\chi(M_{2}\smallsetminus N)=\chi(M_{1}\smallsetminus\mathcal{V})$
by construction, $\chi(N)=2$ and $\chi(\mathcal{V})=0$ by homotopy
equivalences, and $\chi(\partial N)=0$ since $\dim\partial N=3$.
For the Hirzebruch signature of $M_{2}$, we have the following equations:
\begin{align*}
\tau(M_{1}) & =\tau(M_{1}\smallsetminus\mathcal{V})+\tau(\mathcal{V}),\\
\tau(M_{2}) & =\tau(M_{2}\smallsetminus N)+\tau(N).
\end{align*}
We have $H^{2}(\mathcal{V})=H^{2}(S^{1})=0$, so the restriction of
the cup product to the image of the map $H^{2}(\mathcal{V},\partial\mathcal{V})\rightarrow H^{2}(\mathcal{V})$
is zero, and by the definition of the signature for a manifold with
boundary, $\tau(\mathcal{V})=0$. To determine $\tau(N)$, we consider
the trivial bundle $\mathcal{E}=\mathbb{R}^{2}\times S^{2}$ over
$S^{2}$, and set $N=D(\mathcal{E})$ and $\partial N=S(\mathcal{E})$
(resp. the disk and circle subbundles of $\mathcal{E}\rightarrow S^{2}$).
Using the Künneth formula and the Thom isomorphism, we obtain the
following nonzero cohomology groups: 
\begin{align*}
H^{r}(N,\partial N) & \simeq H^{r}(D(\mathcal{E}),S(\mathcal{E}))\simeq H^{r-2}(S^{2})=\mathbb{R}\text{ for }r=2,4,\\
H^{r}(\partial N) & \simeq\mathbb{\mathbb{R}}\text{ for }r=0,\cdots,3.
\end{align*}
 Since $H^{*}(N)\simeq H^{*}(S^{2})$, the long exact sequence for
relative cohomology then reads 
\[
0\rightarrow H^{1}(\partial N)\overset{\delta}{\rightarrow}H^{2}(N,\partial N)\overset{j}{\rightarrow}H^{2}(N)\rightarrow H^{2}(\partial N)\rightarrow0,
\]
and we see that $\delta$ is injective, and that $j$ is zero, which
means that $\tau(N)=0$. Finally:
\[
\tau(M_{2})=\tau(M_{2}\smallsetminus N)=\tau(M_{1}\smallsetminus\mathcal{V})=\tau(M_{1}),
\]
so that
\[
\tau(M_{2})+\frac{1}{|\Gamma|}(\eta^{A}-\rho)=-\frac{2}{3}.
\]

To illustrate theorem \ref{HTIFB}, consider the space $X_{2}=M_{2}\#l\overline{\mathbb{C}P}^{2}$
with $l\ge1$ and the blow-ups in the interior. Proceeding as in the
second example to determine $\chi(X_{2})$ and $\tau(X_{2})$, we
have:
\begin{align*}
\tau(X_{2})+\frac{1}{|\Gamma|}(\eta^{A}-\rho) & =-\Big(\frac{2}{3}+l\Big),\\
\chi(X_{2}) & =l+3,
\end{align*}
and therefore
\[
\chi(X_{2})-\frac{3}{2}\Big|\tau(X_{2})+\frac{1}{k}\big(\eta^{A}-\rho\big)\Big|=\frac{4-l}{2}.
\]
By theorem \ref{HTIFB}, the space $X_{2}$ cannot admit an Einstein
exact $\mathcal{F}/\mathcal{F}_{c}$-metric if $l>4$.
\end{example}
\appendix

\section{Chern-Simons corrections for $\phi$- and $d$-metrics\label{Apndx1}}

Let $N$ and $F$ be closed compact oriented manifolds such that $\dim N+\dim F=4k-1$
for some $k\ge1$ and $\dim F>0$. We consider a smooth fibration
$F\rightarrow W\overset{\phi}{\rightarrow}N$, where the total space
$W$ is the boundary of a compact manifold $M$ with fixed boundary
defining function $x:M\rightarrow\mathbb{R}_{+}$. In this appendix,
we are interested in the behaviour of the Chern-Simons terms of the
Euler characteristic and the signature of $M$ as $x\rightarrow0$,
assuming that this space is equipped with a fibred boundary or a fibred
cusp metric.

\subsection{Riemannian metrics on the boundary: }

In addition to the fibration structure of $W$, we assume that we
have the following geometric objects:
\begin{itemize}
\item A splitting $TW=T^{V}W\oplus T^{H}W$ into vertical and horizontal
subbundles, where $T^{H}W$ is identified with $\phi^{*}TN$ and $T^{V}W=\ker(d\phi)$
(i.e. a connection on $W$);
\item A Riemannian metric $h\in\Gamma(N,S^{2}T^{*}N)$ on $N$;
\item A family of Riemannian metrics $\{g_{F}(y)\}_{y\in N}$ on the fibre
space $F$ that is smoothly parametrized by the base $N$.
\end{itemize}
We thus have a metric $\phi^{*}h$ on the subbundle $T^{H}W$, and
by interpreting the family $\{g_{F}\}$ as a field of symmetric bilinear
forms $\tau\in\Gamma(W,S^{2}(T^{V}W)^{*})$ , we may then define $g_{W}=\phi^{*}h+\tau$,
that gives a Riemannian submersion $\phi:(W,g_{W})\rightarrow(N,h)$
for which the splitting $T^{V}W\oplus T^{H}W$ is orthogonal.

\subsection{Riemannian metrics on $M$:}

In the upcoming discussion, we work on a collar neighborhood of the
boundary in $M$ that is diffeomorphic to $[0,1[_{x}\times W$. The
local coordinates on $N$ will be denoted by $\{y^{i}\}_{i=1}^{\dim N}$,
and we will write $\{z^{a}\}_{a=1}^{\dim F}$ for those on the fibre
$F$. We will use the following notational convention for tensors
on $M$ near the boundary:
\begin{itemize}
\item The index 0 is reserved for the boundary defining function: $\partial_{0}=\frac{\partial}{\partial x}$;
\item The indices $i,j,k,l\in\{1,\cdots,\dim N\}$ designate variables on
the base of $W$: $\partial_{i}=\frac{\partial}{\partial y^{i}}$;
\item The indices $a,b,c,d\in\{1,\cdots,\dim F\}$ refer to coordinates
on the fiber $F$: $\partial_{a}=\frac{\partial}{\partial z^{a}}$;
\item Greek indices will be used to designate arbitrary indices in the summation
convention: $\{\partial_{\alpha}\}=\{\partial_{0},\partial_{i},\partial_{a}\}$.
\end{itemize}
We now set-up the notations for the metrics that we will be dealing
with, which are smooth metrics on the $\phi$- and $d$-tangent bundles,
$^{\phi}TM$ and $^{d}TM$ (\cite{[DaW]}, section 3).

Let $\tilde{g}_{\phi}=(dx/x^{2})^{2}+(\phi^{*}h/x^{2})+\tau$ be a
\textit{product-type fibred boundary metric}, and $\tilde{g}_{d}=x^{2}\tilde{g}_{\phi}$
the associated \textit{product-type fibred cusp metric} ($\phi$-
and $d$-metrics for short), where $g_{W}=\phi^{*}h+\tau$ is the
submersion metric discussed above. Near $W=\partial M$, we have the
following local coordinate expressions:
\begin{align*}
\tilde{g}_{\phi} & =\frac{dx^{2}}{x^{4}}+\frac{h_{ij}(y)}{x^{2}}dy^{i}\otimes dy^{j}+\tau_{ab}(y,z)dz^{a}\otimes dz^{b},\\
\tilde{g}_{d} & =\frac{dx^{2}}{x^{2}}+h_{ij}(y)dy^{i}\otimes dy^{j}+x^{2}\tau_{ab}(y,z)dz^{a}\otimes dz^{b}.
\end{align*}

Let $A,B\in\Gamma(M,S^{2}[{}^{\phi}T^{*}M])$ be symmetric bilinear
forms, with $A$ such that $A(x^{2}\partial_{x},\cdot)\equiv0$ and
$A(x\partial_{i},x\partial_{j})=O(x)$. \textit{Asymptotic metrics}
are first order perturbations (in $x$) of product-type metrics, and
will be denoted by:
\[
\hat{g}_{\phi}=\tilde{g}_{\phi}+x\cdot A\text{, }\hat{g}_{d}=x^{2}\hat{g}_{\phi}.
\]
\textit{Exact metrics} are second order perturbations of product-type
metrics, and will be denoted by
\[
g_{\phi}=\tilde{g}_{\phi}+x\cdot A+x^{2}\cdot B=\hat{g}_{\phi}+x^{2}\cdot B\text{, }g_{d}=x^{2}g_{\phi}.
\]
For asymptotic metrics, we still have an orthogonal decomposition
$^{\phi}T([0,1[_{x}\times W)=\langle x^{2}\partial_{x}\rangle\oplus TW$
near the boundary. From now onward, asymptotic $\phi$-metrics will
be expressed as:
\[
\hat{g}_{\phi}=\frac{dx^{2}}{x^{4}}+\frac{\phi^{*}h}{x^{2}}+\kappa\text{ and }\hat{g}_{d}=\frac{dx^{2}}{x^{2}}+\phi^{*}h+x^{2}\kappa,
\]
where $\kappa=\tau+xA\in\Gamma(W,S^{2}T^{*}W)$ is a bilinear form
on the boundary, depending smoothly on the bdf $x$, and restricting
to a metric on the fibres $F$. As for exact metrics, these are the
most general smooth Riemannian metrics on $^{\phi}TM$ and $^{d}TM$
that we consider here (for a fixed bdf $x$), but all the properties
we are interested in are coming from their ``asymptotic part''.

For exact $\phi$- and $d$-metrics, the Levi-Civita covariant derivatives
will be denoted by $^{\phi}\nabla$ and $^{d}\nabla,$ the connection
1-forms by $^{\phi}\omega$ and $^{d}\omega$, and the curvature 2-forms
by $^{\phi}\Omega$ and $^{d}\Omega$. To designate the correponding
objects associated to the product-type and asymptotic metrics, we
will use the same superscripts on the left and add a ``tilde'' (product)
or a ``hat'' (asymptotic) above.
\begin{rem*}
1) The factor $x^{2}$ in exact metrics is the smallest exponent for
$x$ that gives a well-defined covariant derivative $^{\phi}\nabla:\Gamma(\phantom{}^{\phi}TM)\rightarrow\Gamma(\phantom{}^{\phi}TM\otimes T^{*}M)$.
Indeed, taking $g_{\phi}=\hat{g}_{\phi}+xB$ with $B\in\Gamma(M,S^{2}[^{\phi}T^{*}M])$
would give $\langle dz^{a},\phantom{}^{\phi}\nabla_{\partial_{x}}\partial_{z^{b}}\rangle=O(x^{-1})$,
which blows-up as $x\rightarrow0$.

2) In the local frame given by $\{\partial_{\alpha}\otimes dx^{\alpha}\}$
on $\text{End}TM$, an element $B\in\Gamma(M,S^{2}[^{\phi}T^{*}M])$
decomposes as 
\[
B=\left(\begin{array}{ccc}
\frac{1}{x^{4}}B_{00} & \frac{1}{x^{3}}B_{0i} & \frac{1}{x^{2}}B_{0a}\\
\frac{1}{x^{3}}B_{i0} & \frac{1}{x^{2}}B_{ij} & \frac{1}{x}B_{ia}\\
\frac{1}{x^{2}}B_{a0} & \frac{1}{x}B_{ai} & B_{ab}
\end{array}\right),
\]
where the $B_{\alpha\beta}$ are all smooth on $M$. 
\end{rem*}
Finally, we introduce the auxiliary metrics $\hat{g}_{\varepsilon}$
and $g_{\varepsilon}$ on $TM$ with $\varepsilon\in]0,1[$. The first
one is of asymptotic type:
\[
\hat{g}_{\varepsilon}:=\frac{dx^{2}}{\varepsilon^{4}}+\frac{\phi^{*}h}{\varepsilon^{2}}+\kappa,
\]
while the second is a product metric near the boundary:
\[
g_{\varepsilon}:=\frac{dx^{2}}{\varepsilon^{4}}+\frac{\phi^{*}h}{\varepsilon^{2}}+\tau.
\]
Their restrictions to $W=\partial M$ blow-up the metrics $(\phi^{*}h+\kappa)$
and $g_{W}=(\phi^{*}h+\tau)$ resp. in the direction of the base,
and they coincide with $\phi$-metrics on hypersurfaces $\{x=\varepsilon\}\subset M$:
\[
(\hat{g}_{\varepsilon})_{|\{x=\varepsilon\}}=(\hat{g}_{\phi})_{|\{x=\varepsilon\}}\text{; }(g_{\varepsilon})_{|\{x=\varepsilon\}}=(\tilde{g}_{\phi})_{|\{x=\varepsilon\}}.
\]
Auxiliary metrics are introduced to apply the Atiyah-Patodi-Singer
index theorem. The symbols $^{\varepsilon}\nabla$, $^{\varepsilon}\omega$
and $^{\varepsilon}\Omega$ will respectively denote the Levi-Civita
connection, the connection 1-form and the curvature 2-form of $g_{\varepsilon}$
and $\varepsilon^{2}\cdot g_{\varepsilon}$, while $^{\varepsilon}\hat{\nabla}$,
$^{\varepsilon}\hat{\omega}$ and $^{\varepsilon}\widehat{\Omega}$
will designate the same objects for $\hat{g}_{\varepsilon}$ and $\varepsilon^{2}\cdot\hat{g}_{\varepsilon}$.

We now have an important technical result, that relates the connection
1-forms of the metrics defined above:
\begin{lem}
\label{DiffConnectForms}Let $M_{\varepsilon}=\{x\ge\varepsilon\}\subset M$
with boundary $\partial M_{\varepsilon}=\{x=\varepsilon\}$ for $0<\varepsilon\ll1$.
Then:
\[
(\phantom{}^{\phi}\omega-\phantom{}^{\phi}\hat{\omega})_{|\partial M_{\varepsilon}}\in\varepsilon\cdot\Omega^{1}\left(\partial M_{\varepsilon},\text{End}\left(\phantom{}^{\phi}TM\right)|_{\partial M_{\varepsilon}}\right),
\]
\[
(\phantom{}^{d}\omega-\phantom{}^{d}\hat{\omega})_{|\partial M_{\varepsilon}}\in\varepsilon\cdot\Omega^{1}\left(\partial M_{\varepsilon},\text{End}\left(\phantom{}^{d}TM\right)|_{\partial M_{\varepsilon}}\right).
\]
\[
(\phantom{}^{d}\hat{\omega}-\phantom{}^{\varepsilon}\hat{\omega})_{|\partial M_{\varepsilon}}\in\varepsilon\cdot\Omega^{1}\left(\partial M_{\varepsilon},\text{End}\left(\phantom{}^{d}TM\right)|_{\partial M_{\varepsilon}}\right).
\]
\end{lem}
\begin{rem*}
In all the proofs of this appendix, we will focus on the dependence
on $x$ of the objects involved rather than giving precise expressions.
To obtain the entries of the connection 1-forms, one first needs to
determine the Christoffel symbols of the metrics at hand. The general
procedure for computing these is as follows: Since $\phantom{}^{\phi}TM$
is isomorphic to $TM$ for $x\ne0$, the Christoffel symbols $\Gamma_{\alpha\beta}^{\mu}$
of $g_{\phi}$ with respect to the basis $\{\partial_{\alpha}\}$
are computed by means of the usual formula
\[
\Gamma_{\beta\mu}^{\alpha}=\frac{(g_{\phi}^{-1})^{\alpha\nu}}{2}\left[-\partial_{\nu}(g_{\phi})_{\beta\mu}+\partial_{\beta}(g_{\phi})_{\mu\nu}+\partial_{\mu}(g_{\phi})_{\nu\beta}\right],
\]
and then re-expressed in the basis $\{x^{2}\partial_{x},x\partial_{i},\partial_{a}\}$
or in an orthonormal frame $\{x^{2}\partial_{x},xe_{i},e_{a}\}$ using
the appropriate transformation rule ($\Gamma_{\beta\mu}^{\alpha}$
are not components of a tensor). To obtain the Christoffel symbols
of a $d$-metric, one uses the fact that $g_{d}$ and $g_{\phi}$
are related by a conformal rescaling. For instance, the covariant
derivatives of product-type $\phi$- and $d$-metrics are related
by:
\[
\phantom{}^{d}\widetilde{\nabla}_{X}Y=\frac{1}{x}\cdot\left[\phantom{}^{\phi}\widetilde{\nabla}_{X}(x\cdot Y)\right]+\left(\langle\frac{dx}{x},Y\rangle X-\tilde{g}_{\phi}(X,xY)\cdot\tilde{g}_{\phi}^{-1}\left(\frac{dx}{x^{2}},\cdot\right)\right)
\]
for all $X\in\mathfrak{X}(M)\text{ and }Y\in\Gamma(\phantom{}^{d}TM)=x^{-1}\cdot\Gamma(\phantom{}^{\phi}TM)$,
and one has analogous equations for perturbed metrics.\end{rem*}
\begin{proof}
Let $\Gamma_{\beta\mu}^{\alpha}$ and $\widehat{\Gamma}_{\beta\mu}^{\alpha}$
be the Christoffel symbols of $g_{\phi}$ and $\hat{g}_{\phi}$ in
the basis $\{x^{2}\partial_{x},x\partial_{i},\partial_{a}\}$. A direct
computation yields:
\begin{equation}
(\phantom{}^{\phi}\omega-\phantom{}^{\phi}\hat{\omega})=\left[\left(\Gamma_{\beta\mu}^{\alpha}-\widehat{\Gamma}_{\beta\mu}^{\alpha}\right)dx^{\mu}\right]=\left(\begin{array}{ccc}
0 & 0 & f_{a}^{0}dx\\
0 & 0 & f_{a}^{i}dx\\
f_{0}^{a}dx & f_{i}^{a}dx & f_{b}^{a}dx
\end{array}\right)+x\cdot E,\label{delta_phi}
\end{equation}
where $E\in\Omega^{1}(M,\text{End}(^{\phi}TM))$ and $f_{\beta}^{\alpha}\in\mathcal{C}^{\infty}(M)$
are of order $0$ in $x$, and where we have used the following convention
for the entries of the connection 1-forms:
\[
\phantom{}^{\phi}\omega_{\beta}^{\alpha}=\big<\frac{dx^{\alpha}}{x^{j_{\alpha}}},\phantom{}^{\phi}\nabla_{\partial_{\mu}}(x^{j_{\beta}}\partial_{\beta})\big>\cdot dx^{\mu}=\Gamma_{\beta\mu}^{\alpha}\cdot dx^{\mu},
\]
with $x^{j_{\alpha}}\partial_{\alpha}\text{ and }(dx^{\alpha}/x^{j_{\alpha}})$
being shorthands that designate the fields $x^{2}\partial_{x},x\partial_{i},\partial_{a}\in\Gamma(\phantom{}^{\phi}TM)$
and $(dx/x^{2}),(dy^{i}/x),dz^{a}\in\Gamma(\phantom{}^{\phi}T^{*}M)$.
By restricting to $\partial M_{\varepsilon}=\{x=\varepsilon\}$ in
equation (\ref{delta_phi}), we omit the terms in $dx$ to find that
indeed
\[
(\phantom{}^{\phi}\omega-\phantom{}^{\phi}\hat{\omega})_{|\partial M_{\varepsilon}}\in\varepsilon\cdot\Omega^{1}\left(\partial M_{\varepsilon},\text{End}\left(\phantom{}^{\phi}TM\right)|_{\partial M_{\varepsilon}}\right).
\]
The same computation applied to a $d$-metric proves the second claim,
namely that
\[
(\phantom{}^{d}\omega-\phantom{}^{d}\hat{\omega})_{|\partial M_{\varepsilon}}\in\varepsilon\cdot\Omega^{1}\left(\partial M_{\varepsilon},\text{End}\left(\phantom{}^{d}TM\right)|_{\partial M_{\varepsilon}}\right).
\]
Finally, in a $\hat{g}_{d}$-orthonormal frame $\{x\partial_{x},e_{i},\frac{1}{x}e_{a}\}$
with transition matrix
\[
\Lambda=\left(\begin{array}{ccc}
x & 0 & 0\\
0 & e_{j}^{i} & e_{a}^{i}\\
0 & \frac{1}{x}e_{i}^{a} & \frac{1}{x}e_{b}^{a}
\end{array}\right),
\]
we find that the only nonvanishing entries of $(\phantom{}^{d}\hat{\omega}-\phantom{}^{\varepsilon}\hat{\omega})_{|x=\varepsilon}$
are: 
\[
(\phantom{}^{d}\hat{\omega}-\phantom{}^{\varepsilon}\hat{\omega})_{\alpha}^{0}|_{x=\varepsilon}=-(\phantom{}^{d}\hat{\omega}-\phantom{}^{\varepsilon}\hat{\omega})_{0}^{\alpha}|_{x=\varepsilon}=-\varepsilon\cdot e_{\alpha}^{a}(\kappa_{a\beta})_{|x=\varepsilon}\cdot dx^{\beta}+O(\varepsilon^{2})\text{, }\forall\alpha,\beta\ne0
\]
and the third claim follows.
\end{proof}
The next proposition is used to prove lemma \ref{EtaRhoLmma}:
\begin{prop}
\label{Prop_Symm_bdry_cond}Considering a product metric $\tilde{g}_{\phi}=(dx/x^{2})^{2}+(\phi^{*}h/x^{2})+\tau$
and the asymptotic metric,
\[
\hat{g}_{\phi}=\frac{dx^{2}}{x^{4}}+\frac{\phi^{*}h}{x^{2}}+\kappa=\tilde{g}_{\phi}+xA\text{, }
\]
suppose that $\forall Y\in\mathfrak{X}(N)$ and $\forall Z,V\in\Gamma(M,T^{V}W)$,
the tensor $\kappa=\tau+xA\in\Gamma(W,S^{2}T^{*}W)$ satisfies the
identity:
\begin{equation}
\left(Z\cdot\kappa(V,Y^{H})-V\cdot\kappa(Z,Y^{H})-\kappa([Z,V],Y^{H})\right)_{|x=0}=0,\label{Eq_Symm_bdry_cond}
\end{equation}
where $Y^{H}\in\Gamma(M,T^{H}W)$ is the lift of $Y$. Then the difference
of connection 1-forms $(\phantom{}^{\phi}\hat{\omega}-\phantom{}^{\phi}\widetilde{\omega})$
is such that: 
\[
(\phantom{}^{\phi}\hat{\omega}-\phantom{}^{\phi}\widetilde{\omega})|_{\partial M_{\varepsilon}}\in\varepsilon\cdot\Omega^{1}\left(\partial M_{\varepsilon},\text{End}\left(\phantom{}^{\phi}TM\right)|_{\partial M_{\varepsilon}}\right).
\]
\end{prop}
\begin{proof}
Define $\tilde{\theta}:=(\phantom{}^{\phi}\hat{\omega}-\phantom{}^{\phi}\widetilde{\omega})$,
and write 
\[
\kappa_{ab}=\tau{}_{ab}+x\cdot A_{ab}.
\]
Now consider a local orthonormal frame $\{x^{2}\partial_{x},xe_{i},e_{a}\}\subset\Gamma(^{\phi}TM)$
for $\hat{g}_{\phi}$ with transition matrix
\[
\Lambda=\left(\begin{array}{ccc}
x^{2} & 0 & 0\\
0 & xe_{j}^{i} & xe_{a}^{i}\\
0 & e_{i}^{a} & e_{b}^{a}
\end{array}\right).
\]
A direct computation of the entries of $\tilde{\theta}$ yields that:
\begin{equation}
\tilde{\theta}_{|x=\varepsilon}=\left(\begin{array}{ccc}
0 & 0 & 0\\
0 & \tilde{\theta}_{j}^{i}|_{x=\varepsilon} & \tilde{\theta}_{a}^{i}|_{x=\varepsilon}\\
0 & -\tilde{\theta}_{a}^{i}|_{x=\varepsilon} & 0
\end{array}\right)+\varepsilon\cdot E,\label{Eq_theta_tilde_bndry_1}
\end{equation}
with $E\in\Omega^{1}\left(\partial M_{\varepsilon},\text{End}\left(\phantom{}^{\phi}TM\right)|_{\partial M_{\varepsilon}}\right)$,
and such that for some $C\in\Omega^{1}(\partial M_{\varepsilon})$
and any $\alpha\ne0$: 
\begin{equation}
\tilde{\theta}_{\alpha}^{i}|_{x=\varepsilon}=\frac{1}{2}\Lambda_{i}^{a}\Lambda_{\alpha}^{b}\left(-\partial_{a}\kappa_{bj}+\partial_{b}\kappa_{ja}+\varepsilon\cdot\partial_{j}A_{ab}\right)dy^{j}+\varepsilon\cdot C.\label{Eq_theta_tilde_bndry_2}
\end{equation}

In local coordinates, equation (\ref{Eq_Symm_bdry_cond}) becomes:
\[
\left(\kappa([Z,V],Y^{H})-\left[Z\cdot\kappa(V,Y^{H})-V\cdot\kappa(Z,Y^{H})\right]\right)_{|x=0}=Z^{a}V^{b}Y^{j}(-\partial_{a}\kappa_{bj}+\partial_{b}\kappa_{ja})_{|x=0}=0,
\]
for any $Z^{a},V^{b}\in\mathcal{C}^{\infty}(F)\text{ and }Y^{j}\in\mathcal{C}^{\infty}(N)$,
which simply means that $\forall j\in1,\cdots,\dim N$ and $\forall a,b\in\{1,\cdots,\dim F\}$,
we have $(\partial_{a}\kappa_{bj}-\partial_{b}\kappa_{aj})=O(x)$,
and we get $\tilde{\theta}_{|x=\varepsilon}\in\varepsilon\cdot\Omega^{1}\left(\partial M_{\varepsilon},\text{End}\left(\phantom{}^{\phi}TM\right)|_{\partial M_{\varepsilon}}\right)$.\end{proof}
\begin{rem*}
The condition $A(x\partial_{i},x\partial_{j})=O(x)$ on the tensor
$A\in\Gamma(M,S^{2}[^{\phi}T^{*}M])$ is necessary for the proof above
to work. If $A(x\partial_{i},x\partial_{j})=O(x^{0})$, on the one
hand $\tilde{\theta}_{|x=\varepsilon}\ne O(\varepsilon)$, but more
importantly, we can't apply the Atiyah-Patodi-Singer theorem to $(M_{\varepsilon},\hat{g}_{\varepsilon})$
anymore.
\end{rem*}
A central object in the upcoming computations is the restriction to
$\partial M_{\varepsilon}$ of the curvature $^{\varepsilon}\Omega$
associated to the auxiliary metric $g_{\varepsilon}$. We have the
following fact:
\begin{lem}
\label{Curvature_g_epsilon}Let $^{h}\Omega\in\Omega^{2}(N,\text{End }TN)$
be the connection 2-form of the metric $h$ on the base space $N$,
and let $^{\kappa}\Omega(y)$ be the curvature 2-form of the metric
$\kappa_{|TF_{y}}$ on the fibre $F_{y}=\phi^{-1}(\{y\})\subset W$.
Then:\textup{
\[
^{\varepsilon}\Omega|_{\partial M_{\varepsilon}}=\left(\begin{array}{ccc}
0 & 0 & 0\\
0 & \phi^{*}\left(\phantom{}{}^{h}\Omega\right) & 0\\
0 & 0 & ^{\kappa}\Omega(y)+\alpha(y)
\end{array}\right)+\varepsilon\cdot\left(\begin{array}{cc}
0 & 0\\
0 & E
\end{array}\right),
\]
}where\textup{ }$E\in\Omega^{2}(\partial M_{\varepsilon},\text{End}(T\partial M_{\varepsilon}))$
and $\alpha(y)=\alpha_{ia}dy^{i}\wedge dz^{a}$ locally.\end{lem}
\begin{proof}
First, regarding $\phantom{}^{\varepsilon}\nabla:\Gamma(TM)\rightarrow\Gamma(T^{*}M\otimes TM)$,
we have
\[
\langle dx^{0},\phantom{}^{\varepsilon}\nabla_{\partial_{\beta}}\partial_{\alpha}\rangle=\langle dx^{\mu},\phantom{}^{\varepsilon}\nabla_{\partial_{0}}\partial_{\alpha}\rangle=0,
\]
\[
\langle dx^{\mu},\phantom{}^{\varepsilon}\nabla_{\partial_{\beta}}\partial_{\alpha}\rangle=\left(\mathring{\widetilde{\Gamma}}_{\alpha\beta}^{\mu}\right)_{|x=\varepsilon}\text{, }\forall\mu,\alpha,\beta\ne0.
\]
Once we express $^{\varepsilon}\omega$ in a $g_{\varepsilon}$-orthonormal
frame $\{\varepsilon^{2}\partial_{x},\varepsilon e_{i},e_{a}\}$ near
the boundary $W$, we find:
\[
\phantom{}^{\varepsilon}\omega_{j}^{i}=\Gamma_{jk}^{i}dy^{k}+\varepsilon\gamma_{ja}^{i}dz^{a}\text{, }\phantom{}^{\varepsilon}\omega_{b}^{a}=\varepsilon\gamma_{bk}^{a}dy^{k}+\Gamma_{bc}^{a}(y)dz^{c},
\]
\[
\phantom{}^{\varepsilon}\omega_{a}^{i}=-\phantom{}^{\varepsilon}\omega_{i}^{a}=\varepsilon\gamma_{ib}^{a}dz^{b}\text{, }\phantom{}^{\varepsilon}\omega_{\alpha}^{0}=-\phantom{}^{\varepsilon}\omega_{0}^{\alpha}=0,
\]
with $\gamma_{\beta\mu}^{\alpha}\in\mathcal{C}^{\infty}(M)$ of order
$0$ in $\varepsilon$. We observe that $\Gamma_{jk}^{i}dy^{k}$ give
the entries of the connection 1-form associated to the metric $h$
on $N$, while $\Gamma_{bc}^{a}(y)dz^{c}$ are those of the Levi-Civita
connection form of the metric $\kappa_{|TF_{y}}$on the fibre $F_{y}=\phi^{-1}(\{y\})$.
Using the Maurer-Cartan equation
\[
\phantom{}^{\varepsilon}\Omega=d\phantom{}^{\varepsilon}\omega+\phantom{}^{\varepsilon}\omega\wedge\phantom{}^{\varepsilon}\omega,
\]
one obtains the stated result: 
\[
\phantom{}^{\varepsilon}\Omega_{j}^{i}=\phantom{}^{h}\Omega_{j}^{i}+O(\varepsilon)\text{, }\phantom{}^{\varepsilon}\Omega_{a}^{i}=-\phantom{}^{\varepsilon}\Omega_{i}^{a}=O(\varepsilon),
\]
\[
\phantom{}^{\varepsilon}\Omega_{b}^{a}=\phantom{}^{\kappa}\Omega_{b}^{a}(y)+\partial_{i}\Gamma_{bc}^{a}(y)dy^{i}\wedge dz^{c}+O(\varepsilon),
\]
where:
\begin{align*}
\phantom{}^{h}\Omega_{j}^{i} & =(\partial_{l}\Gamma_{jk}^{i}+\Gamma_{ls}^{i}\Gamma_{jk}^{s})dy^{l}\wedge dy^{k},\\
\phantom{}^{\kappa}\Omega_{b}^{a}(y) & =(\partial_{c}\Gamma_{bd}^{a}(y)+\Gamma_{cf}^{a}(y)\Gamma_{bd}^{f}(y))dz^{c}\wedge dz^{d}.
\end{align*}

\end{proof}

\subsection{Vanishing of Chern-Simons terms: }

The first result here is valid for arbitrary even dimensions $\dim M=2m$.
With the same notations as in lemma \ref{DiffConnectForms}, we have:
\begin{prop}
\label{Prop_CS_prod_exact} Let $\hat{g}_{\phi}$ and $\hat{g}_{d}=x^{2}\hat{g}_{\phi}$
be asymptotic metrics, and for some $B\in\Gamma(M,S^{2}[^{\phi}T^{*}M])$,
consider the exact metrics
\[
g_{\phi}=\hat{g}_{\phi}+x^{2}B\text{ and }g_{d}=\hat{g}_{d}+x^{4}B,
\]
and an asymptotic auxiliary metric $\hat{g}_{\varepsilon}$ on $TM$
such that $(\hat{g}_{\varepsilon})_{|\partial M_{\varepsilon}}\equiv(\hat{g}_{\phi})_{|\partial M_{\varepsilon}}$.
Then for a given invariant polynomial $P\in S^{m}(\mathfrak{so}_{2m}^{*}(\mathbb{R}))$,
the Chern-Simons boundary correction to $P$ obtained from an exact
or an asymptotic $\phi$-metric are the same:
\[
\lim_{\varepsilon\rightarrow0}\int_{\partial M_{\varepsilon}}TP(M,\phantom{}^{\varepsilon}\hat{\nabla},\phantom{}^{\phi}\nabla)=\lim_{\varepsilon\rightarrow0}\int_{\partial M_{\varepsilon}}TP(M,\phantom{}^{\varepsilon}\hat{\nabla},\phantom{}^{\phi}\hat{\nabla}),
\]
and these corrections vanish in the case of $d$-metrics:
\[
\lim_{\varepsilon\rightarrow0}\int_{\partial M_{\varepsilon}}TP(M,\phantom{}^{\varepsilon}\hat{\nabla},\phantom{}^{d}\nabla)=\lim_{\varepsilon\rightarrow0}\int_{\partial M_{\varepsilon}}TP(M,\phantom{}^{\varepsilon}\hat{\nabla},\phantom{}^{d}\hat{\nabla})=0.
\]
\end{prop}
\begin{proof}
Recall that:
\[
P(M,\phantom{}^{\phi}\nabla)\equiv P(\underbrace{\phantom{}^{\phi}\Omega,\cdots,\phantom{}^{\phi}\Omega}_{m\text{ times}})\text{, }TP(M,\phantom{}^{\varepsilon}\hat{\nabla},\phantom{}^{\phi}\nabla)\equiv m\int_{0}^{1}dtP(\phantom{}^{\phi}\omega-\phantom{}^{\varepsilon}\hat{\omega},\underbrace{\phantom{}^{\phi}\Omega_{t},\cdots,\phantom{}^{\phi}\Omega_{t}}_{m-1\text{ times}}),
\]
where $\phantom{}^{\phi}\Omega_{t}$ is the curvature 2-form of the
interpolation connection $^{\phi}\nabla_{t}=\phantom{}^{\varepsilon}\hat{\nabla}+t(\phantom{}^{\phi}\nabla-\phantom{}^{\varepsilon}\hat{\nabla})$.
As in the proof of theorem \ref{Index}, we have 
\[
\int_{\partial M_{\varepsilon}}TP(M,\phantom{}^{\varepsilon}\hat{\nabla},\phantom{}^{\phi}\nabla)=\int_{M_{\varepsilon}}P(M,\phantom{}^{\phi}\nabla)-\int_{M_{\varepsilon}}P(M,\phantom{}^{\varepsilon}\hat{\nabla}),
\]
which leads to 
\begin{align*}
\int_{\partial M_{\varepsilon}}TP(M,\phantom{}^{\varepsilon}\hat{\nabla},\phantom{}^{\phi}\nabla) & =\int_{\partial M_{\varepsilon}}TP(M,\phantom{}^{\varepsilon}\hat{\nabla},\phantom{}^{\phi}\hat{\nabla})+\int_{\partial M_{\varepsilon}}TP(M,\phantom{}^{\phi}\hat{\nabla},\phantom{}^{\phi}\nabla)\\
 & =\int_{\partial M_{\varepsilon}}TP(M,\phantom{}^{\varepsilon}\nabla,\phantom{}^{\phi}\hat{\nabla})\\
 & \phantom{=}+\int_{\partial M_{\varepsilon}}m\left[\int_{0}^{1}P(\phantom{}^{\phi}\omega-\phantom{}^{\phi}\hat{\omega},\phantom{}^{t}\Omega,\cdots,\phantom{}^{t}\Omega)dt\right]_{|\partial M_{\varepsilon}}
\end{align*}
with $^{t}\Omega$ the curvature of the interpolation connection $t\phantom{}^{\phi}\nabla+(1-t)\phantom{}^{\phi}\hat{\nabla}$.
We note that for some $E(t)\in\Omega^{2m-2}(M_{\varepsilon})$:
\[
P(\phantom{}^{\phi}\omega-\phantom{}^{\phi}\hat{\omega},\phantom{}^{t}\Omega,\cdots,\phantom{}^{t}\Omega)\equiv P(\phantom{}^{\phi}\omega-\phantom{}^{\phi}\hat{\omega},\phantom{}^{t}\Omega)=P(\tilde{\theta},\phantom{}^{t}\Omega)_{|x=\varepsilon}+dx\wedge E(t),
\]
so that 
\[
\left[\int_{0}^{1}P(\phantom{}^{\phi}\omega-\phantom{}^{\phi}\hat{\omega},\phantom{}^{t}\Omega,\cdots,\phantom{}^{t}\Omega)dt\right]_{|\partial M_{\varepsilon}}=\int_{0}^{1}P(\tilde{\theta},\phantom{}^{t}\Omega)_{|x=\varepsilon}dt,
\]
and by lemma \ref{DiffConnectForms}: 
\[
\int_{\partial M_{\varepsilon}}TP(M,\phantom{}^{\varepsilon}\hat{\nabla},\phantom{}^{\phi}\nabla)=\int_{\partial M_{\varepsilon}}TP(M,\phantom{}^{\varepsilon}\nabla,\phantom{}^{\phi}\widetilde{\nabla})+\varepsilon\cdot\int_{\partial M_{\varepsilon}}Q.
\]
Since this is also valid for $\phantom{}^{d}\hat{\nabla}$ and $\phantom{}^{d}\nabla$,
we get:
\[
\lim_{\varepsilon\rightarrow0}\int_{\partial M_{\varepsilon}}TP(M,\phantom{}^{\varepsilon}\hat{\nabla},\phantom{}^{\phi}\nabla)=\lim_{\varepsilon\rightarrow0}\int_{\partial M_{\varepsilon}}TP(M,\phantom{}^{\varepsilon}\hat{\nabla},\phantom{}^{\phi}\hat{\nabla}),
\]
and 
\[
\lim_{\varepsilon\rightarrow0}\int_{\partial M_{\varepsilon}}TP(M,\phantom{}^{\varepsilon}\hat{\nabla},\phantom{}^{d}\nabla)=\lim_{\varepsilon\rightarrow0}\int_{\partial M_{\varepsilon}}TP(M,\phantom{}^{\varepsilon}\hat{\nabla},\phantom{}^{d}\hat{\nabla}).
\]
On the other hand, we also have $(\phantom{}^{d}\hat{\omega}-\phantom{}^{\varepsilon}\hat{\omega})_{|x=\varepsilon}=O(\varepsilon)$
by lemma \ref{DiffConnectForms}, so that $TP(M,\phantom{}^{\varepsilon}\nabla,\phantom{}^{d}\hat{\nabla})=\varepsilon Q$
for some $Q\in\Omega^{m-1}(\partial M)$, and the vanishing of the
last two limits above follows.
\end{proof}
The next proposition links the Chern-Simons corrections of exact and
product-type $\phi$-metrics:
\begin{prop}
\label{Prop_CS_phi_bry_cond}Let $M$ be a $2m$-dimensional manifold
with fibred boundary and $P\in S^{m}(\mathfrak{so}_{2m}^{*}(\mathbb{R}))$
an invariant polynomial. Let $\tilde{g}_{\phi}$ and $g_{\varepsilon}$
be product metrics on $^{\phi}TM$ and $TM$ resp., and consider the
metrics
\[
g_{\phi}=\tilde{g}_{\phi}+xA+x^{2}B\text{ and }\hat{g}_{\varepsilon}=g_{\varepsilon}+\varepsilon\cdot A_{|x=\varepsilon},
\]
with $B\in\Gamma(S^{2}[^{\phi}T^{*}M])$, and \textup{$A\in\Gamma(S^{2}[^{\phi}T^{*}M])$}
a symmetric bilinear form such that
\begin{align}
\text{(i) } & A(x^{2}\partial_{x},\cdot)\equiv0\text{ and }A(xY_{1},xY_{2})=O(x),\nonumber \\
\text{(ii) } & \left(V_{1}\cdot A(V_{2},xY_{1})-V_{2}\cdot A(V_{1},xY_{1})-A([V_{1},V_{2}],xY_{1})\right)_{|x=0}=0,\label{Eq_bndry_cond_xA}
\end{align}
for all $Y_{1},Y_{2}\in\mathfrak{X}(N)$ and all $V_{1},V_{2}\in\Gamma(W,T^{V}W)$.
One then has:
\[
\lim_{\varepsilon\rightarrow0}\int_{\partial M_{\varepsilon}}TP(M,\phantom{}^{\varepsilon}\hat{\nabla},\phantom{}^{\phi}\nabla)=\lim_{\varepsilon\rightarrow0}\int_{\partial M_{\varepsilon}}TP(M,\phantom{}^{\varepsilon}\nabla,\phantom{}^{\phi}\widetilde{\nabla}).
\]
\end{prop}
\begin{rem*}
Let $\kappa=x\cdot A$. Condition (i) in this statement means that
$\kappa\in\Gamma(W,S^{2}T^{*}W)$, while condition (ii) is equivalent
to equation (\ref{Eq_Symm_bdry_cond}). Also, we identified the vectors
$Y\in\mathfrak{X}(N)$ with their horizontal lifts to $W$.\end{rem*}
\begin{proof}
With the usual notations for the Levi-Civita connections, we start
by writing: 
\begin{align*}
\left[P(M,\phantom{}^{\phi}\nabla)-P(M,\phantom{}^{\varepsilon}\hat{\nabla})\right] & =\left[P(M,\phantom{}^{\phi}\nabla)-P(M,\phantom{}^{\phi}\hat{\nabla})\right]+\left[P(M,\phantom{}^{\phi}\hat{\nabla})-P(M,\phantom{}^{\phi}\widetilde{\nabla})\right]\\
 & \phantom{=}-\left[P(M,\phantom{}^{\varepsilon}\hat{\nabla})-P(M,\phantom{}^{\varepsilon}\nabla)\right]+\left[P(M,\phantom{}^{\phi}\widetilde{\nabla})-P(M,\phantom{}^{\varepsilon}\nabla)\right],
\end{align*}
then, integrating on $M_{\varepsilon}$ and applying Stokes theorem,
we have for some $E(t)\in\Omega^{2m-2}(\partial M_{\varepsilon})$
and curvature forms $\phantom{}^{\phi}\widehat{\Omega}_{t}\text{, }\phantom{}^{\phi}\widetilde{\Omega}_{t}\text{ and }\phantom{}^{\varepsilon}\Omega_{t}$
of appropriate interpolation connections that:
\begin{align*}
\int_{\partial M_{\varepsilon}}TP(M,\phantom{}^{\varepsilon}\hat{\nabla},\phantom{}^{\phi}\nabla)-\int_{\partial M_{\varepsilon}}TP(M,\phantom{}^{\varepsilon}\nabla,\phantom{}^{\phi}\widetilde{\nabla})\\
=\int_{\partial M_{\varepsilon}}m\Bigg[\int_{0}^{1}dt\cdot\Big\{ dx\wedge E(t)+ & P(\phantom{}^{\phi}\omega-\phantom{}^{\phi}\hat{\omega},\phantom{}^{\phi}\widehat{\Omega}_{t})_{|x=\varepsilon}\\
\phantom{=\int_{\partial M_{\varepsilon}}m\int_{0}^{1}dt\cdot}+P(\phantom{}^{\phi}\hat{\omega}-\phantom{}^{\phi}\widetilde{\omega},\phantom{}^{\phi}\widetilde{\Omega}_{t})_{|x=\varepsilon} & -P(\phantom{}^{\varepsilon}\hat{\omega}-\phantom{}^{\varepsilon}\omega,\phantom{}^{\varepsilon}\Omega_{t})_{|x=\varepsilon}\Big\}\Bigg]_{\big|\partial M_{\varepsilon}},
\end{align*}
by lemma \ref{DiffConnectForms}, we have $(\phantom{}^{\phi}\omega-\phantom{}^{\phi}\hat{\omega})_{|x=\varepsilon}\in\varepsilon\cdot\Omega^{1}(\partial M_{\varepsilon},\text{End}\left(\phantom{}^{\phi}TM\right)|_{\partial M_{\varepsilon}})$,
and by proposition \ref{Prop_Symm_bdry_cond}:
\[
(\phantom{}^{\phi}\hat{\omega}-\phantom{}^{\phi}\widetilde{\omega})|_{x=\varepsilon},(\phantom{}^{\varepsilon}\hat{\omega}-\phantom{}^{\varepsilon}\omega)|_{x=\varepsilon}\in\varepsilon\cdot\Omega^{1}\left(\partial M_{\varepsilon},\text{End}\left(\phantom{}^{\phi}TM\right)|_{\partial M_{\varepsilon}}\right),
\]
which means that for some $Q\in\Omega^{2m-1}(\partial M_{\varepsilon})$,
we get:
\[
\int_{\partial M_{\varepsilon}}TP(M,\phantom{}^{\varepsilon}\hat{\nabla},\phantom{}^{\phi}\nabla)-\int_{\partial M_{\varepsilon}}TP(M,\phantom{}^{\varepsilon}\nabla,\phantom{}^{\phi}\widetilde{\nabla})=\varepsilon\cdot\int_{\partial M_{\varepsilon}}Q,
\]
and the result follows once we take the limit as $\varepsilon\rightarrow0$.
\end{proof}
The last result focuses on the cases needed for our Hitchin-Thorpe
inequality, namely when $P\in S^{m}(\mathfrak{so}_{2m}^{*}(\mathbb{R}))$
is the Hirzebruch $L$-polynomial or the Pfaffian.
\begin{prop}
\label{Prop_CS_vanishing}Let $M$ be compact a $4k$-manifold with
fibred boundary $F\rightarrow\partial M\overset{\phi}{\rightarrow}N$
and consider a splitting $TW=T^{V}W\oplus T^{H}W$ for $W=\partial M$.\\
Let $g_{\phi}$ be an exact metric on $M$ of the form
\[
g_{\phi}=\tilde{g}_{\phi}+x\cdot A+x^{2}\cdot B\text{, with }A,B\in\Gamma(M,S^{2}[^{\phi}T^{*}M]),
\]
where $\tilde{g}_{\phi}$ is a product metric, and $A$ satisfies
the following: 
\begin{align*}
\text{(i) } & A(x^{2}\partial_{x},\cdot)\equiv0\text{ and }A(xY_{1},xY_{2})=O(x),\\
\text{(ii) } & \left(V_{1}\cdot A(V_{2},xY_{1})-V_{2}\cdot A(V_{1},xY_{1})-A([V_{1},V_{2}],xY_{1})\right)_{|x=0}=0,
\end{align*}
for all $Y_{1},Y_{2}\in\mathfrak{X}(N)$ and all $V_{1},V_{2}\in\Gamma(W,T^{V}W)$.

If the dimension of the fibre $F$ is odd, and if $\hat{g}_{\varepsilon}$
is an asymptotic auxiliary metric such that $(\hat{g}_{\varepsilon})_{|\partial M_{\varepsilon}}\equiv(\tilde{g}_{\phi}+xA)_{|\partial M_{\varepsilon}}$,
then the Chern-Simons correction associated to the Euler form vanishes:
\begin{align*}
e_{CS}(\partial M,\phantom{}^{\phi}\nabla)= & -\lim_{\varepsilon\rightarrow0}\int_{\partial M_{\varepsilon}}Te(M,\phantom{}^{\varepsilon}\hat{\nabla},\phantom{}^{\phi}\nabla)=0.
\end{align*}
Furthermore, when $F$ is one-dimensional, the Chern-Simons correction
associated to the Hirzebruch $L$-polynomial also vanishes:
\[
L_{CS}(\partial M,\phantom{}^{\phi}\nabla)=-\lim_{\varepsilon\rightarrow0}\int_{\partial M_{\varepsilon}}TL(M,\phantom{}^{\varepsilon}\hat{\nabla},\phantom{}^{\phi}\nabla)=0.
\]
\end{prop}
\begin{proof}
It is sufficient to consider $g_{\phi}=\tilde{g}_{\phi}$ and $\hat{g}_{\varepsilon}=\tilde{g}_{\varepsilon}$
by proposition \ref{Prop_CS_phi_bry_cond}, and we thus carry-out
the computations in a $\tilde{g}_{\phi}$-orthonormal frame $\{x^{2}\partial_{x},xe_{i},e_{a}\}$
of $^{\phi}TM$ near the boundary (using the same notations as above
for indices, decompositions of matrices, etc). Defining $\tilde{\theta}:=(\phantom{}^{\phi}\widetilde{\omega}-\phantom{}^{\varepsilon}\omega)$,
we employ the interpolation connection $\phantom{}^{\phi}\widetilde{\nabla}_{t}=^{\varepsilon}\nabla+t\tilde{\theta}$,
and write $\phantom{}^{\phi}\widetilde{\omega}_{t}$ and $\phantom{}^{\phi}\widetilde{\Omega}_{t}$
for its connection and curvature forms. On $\partial M_{\varepsilon}=\{x=\varepsilon\}$,
the only nonvanishing entries of $\tilde{\theta}|_{x=\varepsilon}$
are found to be
\begin{equation}
\tilde{\theta}_{i}^{0}|_{x=\varepsilon}=-\tilde{\theta}_{0}^{i}|_{x=\varepsilon}=dy^{i}\text{ and }\tilde{\theta}_{\beta}^{\alpha}|_{x=\varepsilon}=0\text{ }\forall(\alpha,\beta)\ne(0,i)\label{Eq_theta_tilde}
\end{equation}
and since $\phantom{}^{\phi}\widetilde{\omega}_{t}=\phantom{}^{\varepsilon}\omega+t\tilde{\theta}$,
the Maurer-Cartan equation yields: 
\begin{align*}
\phantom{}^{\phi}\widetilde{\Omega}_{t} & =d(\phantom{}^{\varepsilon}\omega+t\tilde{\theta})+(\phantom{}^{\varepsilon}\omega+t\tilde{\theta})\wedge(\phantom{}^{\varepsilon}\omega+t\tilde{\theta})\\
 & =\left(d\phantom{}^{\varepsilon}\omega+\phantom{}^{\varepsilon}\omega\wedge\phantom{}^{\varepsilon}\omega\right)+t(d\tilde{\theta}+\phantom{}^{\varepsilon}\omega\wedge\tilde{\theta}+\tilde{\theta}\wedge\phantom{}^{\varepsilon}\omega)+t^{2}(\tilde{\theta}\wedge\tilde{\theta}),
\end{align*}
so that by lemma \ref{Curvature_g_epsilon}: 
\begin{equation}
\phantom{}^{\phi}\widetilde{\Omega}_{t}|_{x=\varepsilon}=\left(\begin{array}{ccc}
0 & t[dy^{k}\wedge\phantom{}^{\varepsilon}\omega_{k}^{i}] & O(\varepsilon)\\
-t[dy^{k}\wedge\phantom{}^{\varepsilon}\omega_{k}^{i}] & \phi^{*}\left(\phantom{}{}^{h}\Omega\right)+t^{2}[\delta_{kj}dy^{i}\wedge dy^{j}]+O(\varepsilon) & O(\varepsilon)\\
O(\varepsilon) & O(\varepsilon) & ^{\kappa}\Omega(y)+\alpha(y)+O(\varepsilon)
\end{array}\right).\label{Eq_interp_curv}
\end{equation}
\textit{Pfaffian:} For a $2m$-dimensional manifold, the Euler form
is given by 
\[
e(M,\phantom{}^{\phi}\widetilde{\nabla})\equiv Pf\left(\frac{\phantom{}^{\phi}\widetilde{\Omega}}{2\pi}\right)=\frac{c_{m}}{(2m)!}\sum_{\sigma\in S_{2m}}(-1)^{|\sigma|}\phantom{}^{\phi}\widetilde{\Omega}_{\sigma(1)}^{\sigma(2)}\wedge\cdots\wedge\phantom{}^{\phi}\widetilde{\Omega}_{\sigma(2m-1)}^{\sigma(2m)},
\]
for a constant $c_{m}\in\mathbb{R}\smallsetminus\{0\}$. If we write
\begin{align*}
\int_{\partial M_{\varepsilon}}Te(M,\phantom{}^{\varepsilon}\nabla,\phantom{}^{\phi}\widetilde{\nabla}) & =\int_{\partial M_{\varepsilon}}\left[\int_{0}^{1}\left(P(\phantom{}^{\phi}\widetilde{\omega}-\phantom{}^{\varepsilon}\omega,\phantom{}^{\phi}\widetilde{\Omega}_{t})_{|\partial M_{\varepsilon}}+dx\wedge E(t)\right)dt\right]_{|\partial M_{\varepsilon}}\\
 & =\int_{\partial M_{\varepsilon}}\int_{0}^{1}P(\widetilde{\theta},\phantom{}^{\phi}\widetilde{\Omega}_{t})_{|\partial M_{\varepsilon}}dt,
\end{align*}
where $E(t)\in\Omega^{2m-2}(M_{\varepsilon})$, and
\begin{align}
P(\tilde{\theta},\phantom{}^{\phi}\widetilde{\Omega}_{t})_{|\partial M_{\varepsilon}} & =\frac{c_{m}}{(2m)!}\sum_{\sigma\in S_{2m}}(-1)^{|\sigma|}\left[\tilde{\theta}_{\sigma(1)}^{\sigma(2)}\wedge\left(\phantom{}^{\phi}\widetilde{\Omega}_{t}\right)_{\sigma(3)}^{\sigma(4)}\wedge\cdots\wedge\left(\phantom{}^{\phi}\widetilde{\Omega}_{t}\right)_{\sigma(2m-1)}^{\sigma(2m)}\right]_{\big|x=\varepsilon}.\label{Eq_integ_Te}
\end{align}
Now for $\dim F=2f+1$ and $\dim N=2n$, the last expression combined
with equations (\ref{Eq_integ_Te}), (\ref{Eq_theta_tilde}) and (\ref{Eq_interp_curv})
yield that the nonvanishing summands which are proportional to the
volume form on $\partial M_{\varepsilon}$ necessarily come from products
such as:
\begin{equation}
\tilde{\theta}_{0}^{i_{1}}\wedge\underbrace{\left[\left(\phantom{}^{\phi}\widetilde{\Omega}_{t}\right)_{i_{2}}^{i_{3}}\wedge\cdots\wedge\left(\phantom{}^{\phi}\widetilde{\Omega}_{t}\right)_{i_{2n-2}}^{i_{2n-1}}\right]}_{\in\Omega^{2n-2}(N)\oplus\varepsilon\Omega^{2n-2}(\partial M_{\varepsilon})}\wedge\underbrace{\left(\phantom{}^{\phi}\widetilde{\Omega}_{t}\right)_{i_{2n}}^{a_{1}}}_{\propto\varepsilon(dy^{i}\wedge dz^{a})}\wedge\underbrace{\left[\left(\phantom{}^{\phi}\widetilde{\Omega}_{t}\right)_{a_{2}}^{a_{3}}\wedge\cdots\wedge\left(\phantom{}^{\phi}\widetilde{\Omega}_{t}\right)_{a_{2f}}^{a_{2f+1}}\right]}_{\in\Omega^{2f}(F)\oplus\Omega^{2f}(\partial M_{\varepsilon})}.\label{Eq_Te_odd_F}
\end{equation}
where the $i_{j}$'s are indices coming from the base $N$ and the
$a_{j}$'s are associated to the fibre $F$. Again by equation (\ref{Eq_interp_curv}),
the presence of factors $\left(\phantom{}^{\phi}\widetilde{\Omega}_{t}|_{\partial M_{\varepsilon}}\right)_{i_{2n}}^{a_{1}}\in\varepsilon\cdot\Omega^{2}(\partial M_{\varepsilon})$
for odd dimensional fibres leads to:
\[
e_{CS}(\partial M,\phantom{}^{\phi}\nabla)=-\lim_{\varepsilon\rightarrow0}\int_{\partial M_{\varepsilon}}Te(M,\phantom{}^{\varepsilon}\nabla,\phantom{}^{\phi}\widetilde{\nabla})=\lim_{\varepsilon\rightarrow0}\left[\varepsilon\cdot\int_{\partial M_{\varepsilon}}Q\right]=0,
\]
for some $Q\in\Omega^{2m-1}(\partial M)$.

\textit{$L$-polynomial:} Let $M$ be $4k$-dimensional with one-dimensional
fibres of the boundary. If we write
\begin{align*}
\int_{\partial M_{\varepsilon}}TL(M,\phantom{}^{\varepsilon}\nabla,\phantom{}^{\phi}\widetilde{\nabla}) & =\int_{\partial M_{\varepsilon}}\int_{0}^{1}P(\tilde{\theta},\phantom{}^{\phi}\widetilde{\Omega}_{t})|_{x=\varepsilon}dt,
\end{align*}
then the summands of $P(\tilde{\theta},\phantom{}^{\phi}\widetilde{\Omega}_{t})|_{x=\varepsilon}$
that are proportional to a volume form on $\partial M_{\varepsilon}$
are obtained from products of the form:
\[
dy^{i_{1}}\wedge\underbrace{\left[\left(\phantom{}^{\phi}\widetilde{\Omega}_{t}\right)_{i_{2}}^{0}\wedge\left(\phantom{}^{\phi}\widetilde{\Omega}_{t}\right)_{i_{3}}^{i_{4}}\wedge\cdots\wedge\left(\phantom{}^{\phi}\widetilde{\Omega}_{t}\right)_{i_{2n-2}}^{i_{2n-1}}\right]}_{\in\Omega^{2n-2}(N)\oplus\varepsilon\Omega^{2n-2}(\partial M_{\varepsilon})}\wedge\underbrace{\left(\phantom{}^{\phi}\widetilde{\Omega}_{t}\right)_{i_{2n}}^{a}}_{\propto\varepsilon(dy^{k}\wedge dz^{a})},
\]
and
\[
dy^{i_{1}}\wedge\underbrace{\left(\phantom{}^{\phi}\widetilde{\Omega}_{t}\right)_{a}^{0}}_{\propto\varepsilon(dy^{k}\wedge dz^{a})}\wedge\underbrace{\left[\left(\phantom{}^{\phi}\widetilde{\Omega}_{t}\right)_{i_{1}}^{i_{2}}\wedge\left(\phantom{}^{\phi}\widetilde{\Omega}_{t}\right)_{i_{3}}^{i_{4}}\wedge\cdots\wedge\left(\phantom{}^{\phi}\widetilde{\Omega}_{t}\right)_{i_{2n-2}}^{i_{2n-1}}\right]}_{\in\Omega^{2n-2}(N)\oplus\varepsilon\Omega^{2n-2}(\partial M_{\varepsilon})}
\]
which means that if $\dim F=1$ then $TL(M,\phantom{}^{\varepsilon}\nabla,\phantom{}^{\phi}\widetilde{\nabla})=\varepsilon Q$
for some $Q\in\Omega^{4k-1}(\partial M_{\varepsilon})$, and the claim
follows.
\end{proof}

\subsection{Counter-examples: }

We discuss some cases in which propositions \ref{Prop_CS_prod_exact}
and \ref{Prop_CS_vanishing} do not hold anymore, as well as the mistakes
in \cite{[DaW]}. We use the same notations as above.

1) As shown in propositions \ref{Prop_CS_phi_bry_cond} and \ref{Prop_Symm_bdry_cond},
the conditions (\ref{Eq_bndry_cond_xA}) on the perturbation term
$A\in\Gamma(S^{2}[^{\phi}T^{*}M])$ for an asymptotic metric are necessary
for the vanishing of the Chern-Simons corrections, and for the equality
\[
\lim_{\varepsilon\rightarrow0}\int_{\partial M_{\varepsilon}}TP(M,\phantom{}^{\varepsilon}\nabla,\phantom{}^{\phi}\hat{\nabla})=\lim_{\varepsilon\rightarrow0}\int_{\partial M_{\varepsilon}}TP(M,\phantom{}^{\varepsilon}\nabla,\phantom{}^{\phi}\widetilde{\nabla}),
\]
to hold. The analog of proposition \ref{Prop_CS_phi_bry_cond} in
\cite{[DaW]} is Lemma 3.7, and claims the same equality, but the
statement doesn't hold without the conditions \ref{Eq_bndry_cond_xA},
and this causes theorem 3.6 of \cite{[DaW]} to be erroneous too.
At the end of p.563, the equation 
\[
S=xS_{1}+dx\otimes S'
\]
is false (in our notations, $S\equiv\tilde{\theta}=\phantom{}^{\phi}\hat{\nabla}-\phantom{}^{\phi}\widetilde{\nabla}$).
The correct expressions are equations (\ref{Eq_theta_tilde_bndry_1})
and (\ref{Eq_theta_tilde_bndry_2}) in the proof of proposition \ref{Prop_Symm_bdry_cond}.

2) In lemma 4.2 of \cite{[DaW]}, the claim is that for an asymptotic
metric $g_{\phi}=\tilde{g}_{\phi}+xA$ with $A\in\Gamma(S^{2}[^{\phi}T^{*}M])$
such that $A(x^{2}\partial_{x},\cdot)\equiv0$, one has.
\[
\lim_{\varepsilon\rightarrow0}\int_{\partial M_{\varepsilon}}TP(M,\phantom{}^{\varepsilon}\nabla,\phantom{}^{\phi}\hat{\nabla})=0.
\]
Here, the conditions (\ref{Eq_bndry_cond_xA}) of proposition \ref{Prop_CS_phi_bry_cond}
along with the restrictions on $\dim F$ in proposition \ref{Prop_CS_vanishing}
are necessary for this statement to be true, as well as for theorem
4.3 in \cite{[DaW]}. The mistake in the proof of lemma 4.2 of \cite{[DaW]}
is the formula
\[
P(\theta,\Omega_{t},\cdots,\Omega_{t})=\pi^{*}(\alpha)+O(\varepsilon)
\]
on page 566, which comes from computational mistakes in the decomposition
of $\Omega_{t}$. The correct decomposition of $\Omega_{t}$ is given
in equation (\ref{Eq_interp_curv}) of the proof of \ref{Prop_CS_vanishing}.

3) If the fibre $F$ is even-dimensional, then $\lim_{\varepsilon\rightarrow0}\int_{\partial M_{\varepsilon}}Te(M,\phantom{}^{\varepsilon}\nabla,\phantom{}^{\phi}\widetilde{\nabla})$
does not vanish in general. Going back to the proof of proposition
\ref{Prop_CS_vanishing}, if we take $\dim F=2f$ and $\dim N=2n+1$
and employ the same notations, we find that the summands in the integrand
of $Te(M,\phantom{}^{\varepsilon}\nabla,\phantom{}^{\phi}\widetilde{\nabla})|_{\partial M_{\varepsilon}}$
that are multiples of a volume form on $\partial M_{\varepsilon}$
come from products such as
\[
\tilde{\theta}_{0}^{i_{1}}\wedge\underbrace{\left[\left(\phantom{}^{\phi}\widetilde{\Omega}_{t}\right)_{i_{2}}^{i_{3}}\wedge\cdots\wedge\left(\phantom{}^{\phi}\widetilde{\Omega}_{t}\right)_{i_{2n}}^{i_{2n+1}}\right]}_{\in\Omega^{2n}(N)\oplus\varepsilon\Omega^{2n}(\partial M_{\varepsilon})}\wedge\underbrace{\left[\left(\phantom{}^{\phi}\widetilde{\Omega}_{t}\right)_{a_{1}}^{a_{2}}\wedge\left(\phantom{}^{\phi}\widetilde{\Omega}_{t}\right)_{a_{3}}^{a_{4}}\wedge\cdots\wedge\left(\phantom{}^{\phi}\widetilde{\Omega}_{t}\right)_{a_{2f-1}}^{a_{2f}}\right]}_{\in\Omega^{2f}(F)\oplus\Omega^{2f}(\partial M_{\varepsilon})},
\]
which implies that in all generality, $Te(M,\phantom{}^{\varepsilon}\nabla,\phantom{}^{\phi}\widetilde{\nabla})|_{\partial M_{\varepsilon}}$
is not of order at least one in $\varepsilon$ (c.f. equation (\ref{Eq_Te_odd_F})). 

4) If $\dim F>1$, then $\lim_{\varepsilon\rightarrow0}\int_{\partial M_{\varepsilon}}TL(M,\phantom{}^{\varepsilon}\nabla,\phantom{}^{\phi}\widetilde{\nabla})$
does not vanish in general. Recall that for a $4k$-dimensional manifold
$(M,\tilde{g}_{\phi})$, the Hirzebruch $L$-polynomial is given by
\[
L(M,\phantom{}^{\phi}\widetilde{\nabla})\equiv L_{k}[p_{1},\cdots,p_{k}]\left(\frac{\phantom{}^{\phi}\widetilde{\Omega}}{2\pi}\right)=\sum_{\alpha}a_{\alpha}\bigwedge_{j=1}^{k}\left(p_{j}(\phantom{}^{\phi}\widetilde{\Omega}/2\pi)\right)^{\wedge\alpha_{j}},
\]
where the multi-indices $\alpha=(\alpha_{1},\cdots,\alpha_{k})\in\mathbb{N}^{k}$
are such that
\[
\alpha_{1}+2\alpha_{2}+\cdots+j\alpha_{j}+\cdots+k\alpha_{k}=k,
\]
and where the coefficients $a_{\alpha}$ are rational, the $p_{j}(\phantom{}^{\phi}\widetilde{\Omega}/2\pi)$
denotes the $j$-th Pontrjagin form, defined as the form of degree
$4j$ in the expansion 
\[
\det\left(\text{Id}_{4k}+\frac{\phantom{}^{\phi}\widetilde{\Omega}}{2\pi}\right)=1+p_{1}\left(\frac{\phantom{}^{\phi}\widetilde{\Omega}}{2\pi}\right)+\cdots+p_{j}\left(\frac{\phantom{}^{\phi}\widetilde{\Omega}}{2\pi}\right)+\cdots+p_{\dim M/2}\left(\frac{\phantom{}^{\phi}\widetilde{\Omega}}{2\pi}\right),
\]
and $L_{k}[p_{1},\cdots,p_{k}]$ is homogeneous of degree $2k$ in
the entries of $(\phantom{}^{\phi}\widetilde{\Omega}/2\pi)$. We look
at the case where $k=2$ ($\dim\partial M=7$) and $\dim F=3$, in
which one has: 
\begin{align*}
p_{1}(\Omega/2\pi) & =-\frac{1}{2(2\pi)^{2}}\text{Tr}(\Omega^{2}),\\
p_{2}(\Omega/2\pi) & =\frac{1}{8(2\pi)^{4}}[\{\text{Tr}(\Omega^{2})\}^{2}-\text{Tr}(\Omega^{4})],
\end{align*}
and 
\[
L_{2}[p_{1},p_{2}](\Omega/2\pi)=\frac{1}{45}[7p_{2}-(p_{1})^{2}](\Omega/2\pi)=\frac{(2\pi)^{-4}}{360}\left[5\left(\text{Tr}(\Omega^{2})\right)^{2}-7\text{Tr}(\Omega^{4})\right].
\]
Using the notations in the proof of proposition \ref{Prop_CS_vanishing},
analyzing the integrand of $TL(M,\phantom{}^{\varepsilon}\nabla,\phantom{}^{\phi}\widetilde{\nabla})|_{\partial M_{\varepsilon}}=\int_{0}^{1}dt[P(\tilde{\theta},\phantom{}^{\phi}\widetilde{\Omega}_{t})|_{\partial M_{\varepsilon}}]$
yields that for some constant $c\ne0$:
\[
P(\tilde{\theta},\phantom{}^{\phi}\widetilde{\Omega}_{t})=c\cdot\underbrace{\left[\sum_{i=1}^{\dim N=4}\tilde{\theta}_{i}^{0}\wedge\left(\phantom{}^{\phi}\widetilde{\Omega}_{t}\right)_{i}^{0}\right]}_{\propto(dy^{i}\wedge dy^{j}\wedge dy^{k})}\wedge\underbrace{\left[\sum_{a,b=1}^{\dim F=3}\left(\phantom{}^{\phi}\widetilde{\Omega}_{t}\right)_{a}^{b}\wedge\left(\phantom{}^{\phi}\widetilde{\Omega}_{t}\right)_{a}^{b}\right]}_{\propto(dy^{l}\wedge dz^{1}\wedge dz^{2}\wedge dz^{3})}+O(\varepsilon),
\]
and by equations (\ref{Eq_theta_tilde}) and (\ref{Eq_interp_curv})
above, $TL(M,\phantom{}^{\varepsilon}\nabla,\phantom{}^{\phi}\widetilde{\nabla})|_{\partial M_{\varepsilon}}$
is not of order 1 in $\varepsilon$, and may not necessarily vanish
when $\varepsilon\rightarrow0$.

\section{Multi Taub-NUT metrics\label{Apndx2}}

As above, we have $k$ points $\{p_{j}\}_{j=1}^{k}$ in $\mathbb{R}^{3}$,
a principal circle bundle $M\overset{\pi}{\longrightarrow}\mathbb{R}^{3}\smallsetminus\{p_{j}\}$
of degree -1 near each $p_{j}$, and we are considering the following
metric on $M$: 
\[
g=\pi^{*}[V\cdot((dx^{1})^{2}+(dx^{2})^{2}+(dx^{3})^{2})]+\pi^{*}(V^{-1})\cdot(d\theta+\pi^{*}\omega)^{2},
\]
where $\theta$ is a coordinate on the fibres, $\omega$ is a connection
1-form on $\mathbb{R}^{3}\smallsetminus\{p_{j}\}$ with curvature
$d\omega=*_{\mathbb{R}^{3}}dV$, and the function $V:\mathbb{R}^{3}\smallsetminus\{p_{j}\}_{j=1}^{k}\rightarrow\mathbb{R}$
is defined as: 
\[
V(x)=1+\frac{1}{2}{\displaystyle {\sum_{j=1}^{k}}\frac{1}{|x-p_{j}|}}.
\]
The harmonic function $V$ determines $g_{0}$ uniquely, since for
any gauge transformation $\omega\mapsto\omega+df$ with $f\in\mathcal{C}^{\infty}(\mathbb{R}^{3}\smallsetminus\{p_{j}\})$,
we can make the change of variable $\theta\mapsto\theta+f$ to keep
the same metric (\cite{[WW]}, section 6.2).

Now we briefly recall the construction of the smooth completion $(M_{0},g_{0})$,
following section 3 of \cite{[ES]} (see also \cite{[AKL]} and \cite{[LeB91]}).
For $0<\delta<\min_{1\le j<i\le k}|p_{i}-p_{j}|$, we consider the
punctured open balls $B_{i}=\mathbb{B}_{\delta}^{3}(p_{i})\smallsetminus\{p_{i}\}\subset\mathbb{R}^{3}$
($i=1,\cdots,k$), and note that $\pi^{-1}(B_{i})$ are diffeomorphic
to $\mathbb{B}_{\delta}^{4}(0)\smallsetminus\{0\}$ in such a way
that the $S^{1}$ action coincides with scalar multiplication on $\mathbb{R}^{4}\simeq\mathbb{C}^{2}$.
We then define
\[
M_{0}:=(M\sqcup\bigsqcup_{j=1}^{k}\mathbb{B}_{\delta}^{4}(q_{i}))/\sim,
\]
where $\sim$ stands for the identification of the $\pi^{-1}(B_{i})$
with punctured 4-balls $\mathbb{B}_{\delta}^{4}(q_{i})\smallsetminus\{q_{i}\}$.
The map $\pi:M\rightarrow\mathbb{R}^{3}\smallsetminus\{p_{j}\}$ is
then smoothly extended to $\pi:M_{0}\rightarrow\mathbb{R}^{3}$ with
$q_{i}=\pi^{-1}(p_{i})$, and $\pi$ acts as the projection to the
space of $S^{1}$-orbits (under scalar multiplication) when restricted
to the balls $\mathbb{B}_{\delta}^{4}(q_{i})$. To see that we can
also extend $g$ to a smooth metric $g_{0}$ on $M_{0}$, we note
that in the vicinity of the points $q_{i}$, we can write $g=g_{i}^{F}+\alpha_{i}$
where $g_{i}^{F}$ is isometric to the flat Euclidean metric on $\mathbb{B}_{\delta}^{4}(q_{i})$,
and $\alpha_{i}$ is a symmetric bilinear form which is smooth on
the same neighbourhood. Indeed, let $(r,\varphi,\psi)$ be the spherical
coordinates centred at $p_{i}\in\mathbb{R}^{3}$, and write $V=V_{i}+f_{i}$
with $V_{i}(x)=(2|x-p_{i}|)^{-1}=1/(2r)$ and $f_{i}$ smooth on $\mathbb{B}_{\delta}^{3}(p_{i})$.
Since 
\[
\ast_{\mathbb{R}^{3}}dV_{i}=-\frac{1}{2}\sin\varphi d\varphi\wedge d\psi=d\left(\frac{\cos\varphi d\psi}{2}\right),
\]
we may take $\omega_{i}(x)=(\cos\varphi d\psi)/2$ as a local connection
1-form for the curvature $\ast dV_{i}$, and the restriction $g_{|\mathbb{B}_{\delta}^{4}(q_{i})\smallsetminus\{q_{i}\}}=g_{i}^{F}+\alpha_{i}$
is then given by: 
\begin{align*}
g_{i}^{F} & =\pi^{*}[V_{i}\cdot((dx^{1})^{2}+(dx^{2})^{2}+(dx^{3})^{2})]+\pi^{*}(V_{i}^{-1})\cdot(d\theta+\pi^{*}\omega_{i})^{2}\\
 & =\frac{(dr)^{2}}{2r}+\frac{r}{2}\left[(d\varphi)^{2}+(d\psi)^{2}+(d(2\theta))^{2}+2\cos\varphi d(2\theta)\odot d\psi\right],\\
\alpha_{i} & =f_{i}\pi^{*}((dx^{1})^{2}+(dx^{2})^{2}+(dx^{3})^{2})+V^{-1}\cdot\big[-(f_{i}/V_{i})(d\theta+\pi^{*}\omega_{i})^{2}\\
 & \phantom{=V^{-1}\cdot\big[}+\pi^{*}(\omega-\omega_{i})^{2}+2(d\theta+\pi^{*}\omega_{i})\odot\pi^{*}(\omega-\omega_{i})\big].
\end{align*}
On $\mathbb{B}_{\delta}^{4}(q_{i})$, we introduce the following coordinates
\cite{[EH78]} :
\begin{align*}
y^{1}+iy^{2} & =\sqrt{2r}\cos\left(\frac{\varphi}{2}\right)\exp\left[i\left(\theta+\frac{\psi}{2}\right)\right]\\
y^{3}+iy^{4} & =\sqrt{2r}\sin\left(\frac{\varphi}{2}\right)\exp\left[i\left(\theta-\frac{\psi}{2}\right)\right],
\end{align*}
and we obtain that $g_{i}^{F}=(dy^{1})^{2}+(dy^{2})^{2}+(dy^{3})^{2}+(dy^{4})^{2}$.
Letting $y\rightarrow q_{i}$ in $M_{0}$, we have $\alpha_{i}\rightarrow f_{i}\pi^{*}((dx^{1})^{2}+(dx^{2})^{2}+(dx^{3})^{2})_{|p_{i}}$,
and it is then obvious that $g=g_{i}^{F}+\alpha_{i}$ can be smoothly
extended to $q_{i}$.

Finally, the condition $d\omega=*_{\mathbb{R}^{3}}dV$ implies that
$g$ is Ricci-flat, and we can define 3 compatible parallel complex
structures $\{J_{l}\}_{l=1}^{3}$ on $M$ by taking
\[
J_{1}:\text{ }\pi^{*}dx^{1}\mapsto V^{-1}(d\theta+\pi^{*}\omega)\text{, }\pi^{*}dx^{2}\mapsto\pi^{*}dx^{3},
\]
and permuting the action on these forms for $J_{2}$ and $J_{3}$.
Since $M_{0}$ is simply connected, we obtain that $g_{0}$ is a complete
Ricci-flat hyper-K\"ahler metric for this space.

\section*{\textbf{Acknowledgements}}

The author is thankful to his M.Sc. advisor F. Rochon, for giving
him the opportunity to work on this problem and for many enlightening
discussions, as well as to an anonymous referee for his many comments
that substantially improved this article (he provided the counter-example
in the remark following theorem \ref{Index}). This work was partially
supported by an FRQNT Doctoral Scholarship.

\bibliography{Biblio}

\begin{thebibliography}{{Bes}08}

\bibitem[AB68]{[AB2]}
M.~F. Atiyah and R.~Bott.
\newblock A {L}efschetz fixed point formula for elliptic complexes. {II}.
  {A}pplications.
\newblock {\em Ann. of Math. (2)}, 88:451--491, 1968.

\bibitem[AKL89]{[AKL]}
Michael~T. Anderson, Peter~B. Kronheimer, and Claude LeBrun.
\newblock Complete {R}icci-flat {K}\"ahler manifolds of infinite topological
  type.
\newblock {\em Comm. Math. Phys.}, 125(4):637--642, 1989.

\bibitem[APS75a]{[APS1]}
M.~F. Atiyah, V.~K. Patodi, and I.~M. Singer.
\newblock Spectral asymmetry and {R}iemannian geometry. {I}.
\newblock {\em Math. Proc. Cambridge Philos. Soc.}, 77:43--69, 1975.

\bibitem[APS75b]{[APS2]}
M.~F. Atiyah, V.~K. Patodi, and I.~M. Singer.
\newblock Spectral asymmetry and {R}iemannian geometry. {II}.
\newblock {\em Math. Proc. Cambridge Philos. Soc.}, 78(3):405--432, 1975.

\bibitem[{Bes}08]{[B2]}
Arthur~L. {Besse}.
\newblock {\em {Einstein manifolds. Reprint of the 1987 edition.}}
\newblock Berlin: Springer, reprint of the 1987 edition edition, 2008.

\bibitem[Don78]{[Don]}
Harold Donnelly.
\newblock Eta invariants for {$G$}-spaces.
\newblock {\em Indiana Univ. Math. J.}, 27(6):889--918, 1978.

\bibitem[DW07]{[DaW]}
Xianzhe Dai and Guofang Wei.
\newblock Hitchin-{T}horpe inequality for noncompact {E}instein 4-manifolds.
\newblock {\em Adv. Math.}, 214(2):551--570, 2007.

\bibitem[DZ95]{[DaZ]}
Xianzhe Dai and Wei~Ping Zhang.
\newblock Circle bundles and the {K}reck-{S}tolz invariant.
\newblock {\em Trans. Amer. Math. Soc.}, 347(9):3587--3593, 1995.

\bibitem[EH79]{[EH78]}
Tohru Eguchi and Andrew~J. Hanson.
\newblock Self-dual solutions to {E}uclidean gravity.
\newblock {\em Ann. Physics}, 120(1):82--106, 1979.

\bibitem[ES11]{[ES]}
G{\'a}bor Etesi and Szil{\'a}rd Szab{\'o}.
\newblock Harmonic functions and instanton moduli spaces on the
  multi-{T}aub--{NUT} space.
\newblock {\em Comm. Math. Phys.}, 301(1):175--214, 2011.

\bibitem[GH79]{[GH]}
G.~W. Gibbons and S.~W. Hawking.
\newblock Classification of gravitational instanton symmetries.
\newblock {\em Comm. Math. Phys.}, 66(3):291--310, 1979.

\bibitem[Gil95]{[G]}
Peter~B. Gilkey.
\newblock {\em Invariance theory, the heat equation, and the {A}tiyah-{S}inger
  index theorem}.
\newblock Studies in Advanced Mathematics. CRC Press, Boca Raton, FL, second
  edition, 1995.

\bibitem[GR15]{[GRR]}
Jesse {Gell-Redman} and Fr\'ed\'eric {Rochon}.
\newblock {Hodge cohomology of some foliated boundary and foliated cusp
  metrics.}
\newblock {\em {Math. Nachr.}}, 288(2-3):206--223, 2015.

\bibitem[Gro82]{[Grom82]}
Michael Gromov.
\newblock Volume and bounded cohomology.
\newblock {\em Inst. Hautes \'Etudes Sci. Publ. Math.}, (56):5--99 (1983),
  1982.

\bibitem[Hit74]{[H]}
Nigel Hitchin.
\newblock Compact four-dimensional {E}instein manifolds.
\newblock {\em J. Differential Geometry}, 9:435--441, 1974.

\bibitem[Hit79]{[H79]}
N.~J. Hitchin.
\newblock Polygons and gravitons.
\newblock {\em Math. Proc. Cambridge Philos. Soc.}, 85(3):465--476, 1979.

\bibitem[Kot98]{[Kot98]}
Dieter Kotschick.
\newblock On the {G}romov-{H}itchin-{T}horpe inequality.
\newblock {\em C. R. Acad. Sci. Paris S\'er. I Math.}, 326(6):727--731, 1998.

\bibitem[Kot12]{[Kot12]}
Dieter Kotschick.
\newblock Entropies, volumes, and {E}instein metrics.
\newblock In {\em Global differential geometry}, volume~17 of {\em Springer
  Proc. Math.}, pages 39--54. Springer, Heidelberg, 2012.

\bibitem[LeB91]{[LeB91]}
Claude LeBrun.
\newblock Complete {R}icci-flat {K}\"ahler metrics on {${\bf C}\sp n$} need not
  be flat.
\newblock In {\em Several complex variables and complex geometry, {P}art 2
  ({S}anta {C}ruz, {CA}, 1989)}, volume~52 of {\em Proc. Sympos. Pure Math.},
  pages 297--304. Amer. Math. Soc., Providence, RI, 1991.

\bibitem[LM89]{[LM]}
H.~Blaine Lawson, Jr. and Marie-Louise Michelsohn.
\newblock {\em Spin geometry}, volume~38 of {\em Princeton Mathematical
  Series}.
\newblock Princeton University Press, Princeton, NJ, 1989.

\bibitem[MM98]{[MaMel98]}
Rafe Mazzeo and Richard~B. Melrose.
\newblock Pseudodifferential operators on manifolds with fibred boundaries.
\newblock {\em Asian J. Math.}, 2(4):833--866, 1998.
\newblock Mikio Sato: a great Japanese mathematician of the twentieth century.

\bibitem[Roc12]{[R]}
Fr{\'e}d{\'e}ric Rochon.
\newblock Pseudodifferential operators on manifolds with foliated boundaries.
\newblock {\em J. Funct. Anal.}, 262(3):1309--1362, 2012.

\bibitem[Sam98]{[Sam98]}
Andrea Sambusetti.
\newblock An obstruction to the existence of {E}instein metrics on
  {$4$}-manifolds.
\newblock {\em Math. Ann.}, 311(3):533--547, 1998.

\bibitem[ST69]{[ST]}
I.~M. Singer and J.~A. Thorpe.
\newblock The curvature of {$4$}-dimensional {E}instein spaces.
\newblock In {\em Global {A}nalysis ({P}apers in {H}onor of {K}. {K}odaira)},
  pages 355--365. Univ. Tokyo Press, Tokyo, 1969.

\bibitem[Vai01]{[Vai01]}
Boris Vaillant.
\newblock {\em Index- and spectral theory for manifolds with generalized fibred
  cusps}.
\newblock Bonner Mathematische Schriften [Bonn Mathematical Publications], 344.
  Universit\"at Bonn, Mathematisches Institut, Bonn, 2001.
\newblock Dissertation, Rheinische Friedrich-Wilhelms-Universit{\"a}t Bonn,
  Bonn, 2001.

\bibitem[Wri12]{[W]}
Evan~P. Wright.
\newblock Quotients of gravitational instantons.
\newblock {\em Ann. Global Anal. Geom.}, 41(1):91--108, 2012.

\bibitem[WW90]{[WW]}
R.~S. Ward and Raymond~O. Wells, Jr.
\newblock {\em Twistor geometry and field theory}.
\newblock Cambridge Monographs on Mathematical Physics. Cambridge University
  Press, Cambridge, 1990.

\end{thebibliography}
 \bibliographystyle{alpha}

\curraddr{Department of Mathematics, University of Toronto}

\email{ajzer@math.toronto.edu}
\end{document}